\documentclass[twoside]{article}

\usepackage{geometry}
\usepackage{scrhack}

\usepackage[utopia]{mathdesign}

\usepackage{dsfont, amsthm}

\newcommand{\SetFigFont}[2]{}

\usepackage[utf8]{inputenc}

\usepackage[american]{babel}

\usepackage{fancyhdr}
\usepackage{color}
\usepackage{array}
\usepackage{booktabs}
\usepackage[final]{graphicx}
\usepackage{enumitem}
\usepackage{url}
\usepackage{varwidth}
\usepackage{hhline}

\usepackage{amsmath}
\usepackage{mathtools}

\usepackage[
  normal,
  font={small,color=black},
  labelfont=bf,
  margin=2em,
  labelsep=period]{caption}[2008/04/01]

\usepackage{caption}

\usepackage[square,numbers,super,sort&compress]{natbib}

\makeatletter
\renewcommand\NAT@citesuper[3]{\ifNAT@swa
\if*#2*\else#2\NAT@spacechar\fi
\unskip\kern\p@\textsuperscript{\NAT@@open#1\if*#3*\else,\NAT@spacechar#3\fi\NAT@@close}%
   \else #1\fi\endgroup}
\makeatother
\usepackage{makeidx}         
\usepackage{graphicx}        
\usepackage{multicol}        
\usepackage[bottom]{footmisc}

\usepackage{array}

\providecommand{\keywords}[1]{\textbf{Keywords} #1}

\newtheorem{lemma}{Lemma}[section]
\newtheorem{remark}[lemma]{Remark}
\newtheorem*{acknowledgement}{Acknowledgement}

\newcommand{\Dt}{\partial_{t}\space}

\newcommand{\vr}{\hat{v}\space}
\newcommand{\Tr}{\hat{T}\space}
\newcommand{\Fr}{\hat{F}\space}
\newcommand{\Fd}{\hat{\Omega}\space}

\newcommand{\pr}{\hat{p}\space}

\newcommand{\phir}{\hat{\phi}\space}

\newcommand{\Jr}{\hat{J}\space}
\newcommand{\sigmar}{\hat{\sigma}\space}
\newcommand{\nablar}{\hat{\nabla}\space}

\newcommand{\fr}{\hat{f}\space}
\newcommand{\xhir}{\hat{\xi}\space}

\begin{document}

\title{Finite element simulation of fluid dynamics and 
CO$_2$ gas exchange in the alveolar sacs of the human lung}

\author{Luis J. Caucha\thanks{Universidad Nacional de Tumbes, Peru 
(ljcaucham@untumbes.edu.pe)} \and
Stefan Frei\thanks{Department of Mathematics, University College London, UK, (s.frei@ucl.ac.uk)}
\and Obidio Rubio\thanks{Universidad Nacional de Trujillo, Peru (orubio@unitru.edu.pe)}}

\date{}
\maketitle

\abstract{In this article we present a numerical framework based on continuum models for the fluid dynamics
and the CO$_2$ gas distribution in the alveolar sacs of the human lung during expiration and inspiration, 
including the gas exchange
to the cardiovascular system. We include the expansion and contraction of the geometry
by means of the Arbitrary Lagrangian Eulerian (ALE) method. For discretisation, we use
equal-order finite elements in combination with pressure-stabilisation techniques
based on local projections or interior penalties. We derive formulations
for both techniques that are suitable on arbitrarily anisotropic meshes. These formulations
are novel within the ALE method. 
Moreover, we investigate the effect of different boundary conditions, that vary between inspiration 
and expiration.
We present numerical results on a
simplified two-dimensional alveolar sac geometry and investigate 
the influence of the pressure stabilisations as well as the boundary conditions.\\}


\section{Introduction}\label{sec:Intro}

 The principal tasks of the human respiratory system are the oxygen uptake and the release
 of carbon dioxide. The exchange of both with the cardiovascular system takes place in the distal zone
 of the human lung, in tiny grape-like structures called alveoli. These are connected to the
 cardiovascular system by the alveolar-capillary membranes through which the carbon dioxide may enter. 
 The alveoli are bundled in alveolar sacs to which the alveolar ducts open distally.
 
 Understanding the dynamics of gases in the alveolar sacs is important to investigate several pathologies
 such as pulmonary emphysema\cite{Parya2016}. Here, foreign substances such as tobacco accumulate in the alveoli
 and cause a partial or full blockage of the membranes. Thus, the alveolar sacs fill up with air, 
 expand and may break.
 
 Today's standard model of the human lung has been proposed by Weibel \cite{Weibel1963}.
 Weibel's  model splits the lung into 23 sub-regions, so called \textit{generations},
 which represent the number of bifurcations 
 starting from the trachea ($0^{\text{th}}$ generation) to the bronchial tree ($23^{\text{rd}}$ generation). 
 Alveoli appear in the bronchial tree between the $15^{\text{th}}$ generation to $23^{\text{rd}}$ generation.
 While the upper airways have a very stiff cartilage structure, the lower airway walls undergo considerable 
 deformations during inspiration and expiration. 
  
A first study of gas distribution in the human lung was presented by 
Milic-Emili et al. \cite{MilicEmili1966} who studied the distribution of Xenon (Xe$^{133}$).
Models to simulate the convective flow in the acinus have gained interest starting
 from investigations about gas mixture and exchange~\cite{Engel1983, Paiva1985}. First simulations
 were made considering ideal conditions in the alveolar zone and a homogeneous gas concentration 
 in each respiration~\cite{Scrimshire1973, Felici2003}.
Federspiel \& Fredberg~\cite{Federspiel1988} made an axial gas dispersion analysis in a model of respiratory
 bronchioles and alveolar ducts, using the Navier-Stokes equations and a convection-diffusion equation for the
 gas transport. 
 
 Later on, different mathematical and numerical models have been developed specifically for the fluid 
 and particle dynamics within the alveolar sacs. 
 Gefen et al.~\cite{Gefen1999} used the finite element method to simulate the stress distribution on the parenchymal
 micro-structure. Suresh et al.~\cite{Sureshetal2005} studied the CO$_2$ gas concentration and gas exchange in one single alveoli,
 using a finite volume method approach with simplified geometries and fluid models.
 
 On the alveolar scale many works are based on kinetic models instead of continuum ones. This is justified, especially if the dynamics of
 larger particles are considered, as typical sizes of the alveolar sacs are roughly between $0.1$ and $0.2mm$~\cite{Sznitman2013}.
 Tsuda et al.~\cite{Tsudaetal1994} used
Monte-Carlo methods to study the diffusional deposition of aerosols in the alveolar ducts. Further improvements towards a realistic simulation
were made by Darquenne and co-workers~\cite{DarquennePaiva1996},\cite{Darquenne2001} and by
Kumar et al.~\cite{Kumaretal2009} using a continuum 
description for the flow part, including simulation results on representative three-dimensional geometries.

All the previously mentioned investigations were made on a fixed domain. During expiration and inspiration, however,
the alveoli grow and shrink considerably (see e.g.\cite{Sznitman2013, Esra2011} and the references cited therein). 
First results on moving domains with a prescribed wall motion 
were given by Henry et al.~\cite{Henryetal2002} and by Haber and 
co-workers~\cite{Haberetal2000, Haberetal2003} using 
kinetic models to study the motion and deposition of particles. Li \& Kleinstreuer~\cite{LiKleinstreuer2011} used a Lattice-Boltzmann 
method to simulate the fluid flow in the alveolar region.
In all these works a strong dependence on the wall motion was found.

Recently, Darquenne and co-workers studied a three-dimensional flow model including their movement and a kinetic model for the particle
flow and deposition~\cite{Darquenneetal2009}. Moreover, we would like to mention the work of Dailey \& Ghadali, who made a fluid-structure-interaction simulation
of the interaction of the airflow inside the alveoli with the surrounding tissue~\cite{DaileyGhadali}, including a variable deformation of the alveolar walls in between. 
The main challenge is to model the surrounding tissue, which might be very heterogeneous and little is known on its elastic properties. The authors 
use a viscoelastic model in combination with an applied exterior
oscillatory pressure and study the effect of different elasticity 
parameters.
Finally we would like to remark that a detailed review with further developments and details has been published by Sznitman~\cite{Sznitman2013}.

In this article, we present a numerical framework to simulate
gas exchange and gas transport in a moving alveolar sac based purely on continuum models. 
We include the effect 
of moving boundaries by imposing the domain movement, as 
the deformation is caused by
the movement of the surrounding diaphragm, which can be 
observed in reality \cite{Wade1954,Gilbert1981}.
Specifically, we are interested in the distribution
of CO$_2$ in the alveoli. The specific gas, however, is exemplary and our framework can be used in a 
straight-forward way to compute the distributions of other substances.

The convective flow in the alveolar sacs is governed by the incompressible Navier-Stokes equations. Gas transport is
described by a convection-diffusion equation. We formulate both equations on a moving domain first and use the 
\textit{Arbitrary Lagrangian Eulerian} (ALE) method~\cite{Donea1977, HuLiZi81} to transform the equations to a fixed domain.
The ALE method can be considered as a standard approach for flow problems on moving
domains.

One particular difficulty for the numerical simulation is the in- and outflow condition for the gas concentration. 
A classical approach is to assume a Dirichlet condition during inflow (inspiration) and a homogeneous Neumann
condition during outflow (expiration). On the alveolar scale the implementation of a Dirichlet condition during inflow is however not straight-forward,
as the exterior gas concentration depends on the outflow during expiration and is not known \textit{a priori}. As an alternative, we
introduce a second type of boundary conditions, based on so-called artificial or transparent boundary conditions~\cite{Ehrhardt, Halpern}.
In both approaches, the boundary conditions are incorporated weakly in the
variational formulation, which avoids a change of the number of unknowns and
structure of the system matrix from one time step to another.

For discretisation, we use equal-order finite elements in combination with pressure stabilisation.
Among the different possibilities for pressure stabilisation, we study two related approaches,
that are frequently used in the literature, namely the \textit{Local Projection Stabilisation} (LPS) technique~\cite{BeckerBraack2001} and
an \textit{interior penalty} (ip) method~\cite{BurmanFernandezHansbo2006}. 
It is well-known that both methods are related, in the sense that if the ip stabilisation is used 
on interior edges of patch grids only, both stabilisation terms are upper- and lower-bounded by each other.
For an accurate and efficient simulation, we will use 
anisotropic mesh cells. Due to the movement of the alveolar sacs during inspiration and expiration these may become even more anisotropic. 
To cope with these difficulties, we derive anisotropic variants of the stabilisation techniques that are 
novel in the context of the ALE approach.

The purpose of this work is the introduction and investigation of a model 
and a numerical framework that can be applied 
for the simulation on arbitrary alveolar geometries.
As the shape of 
the alveolar sacs is very complex, and moreover it varies between different generations,
we concentrate in a first step on two-dimensional model geometries in the numerical results. This allows us
to focus on modelling and discretisation issues, which is the main scope of this paper.
Both the model and the numerical framework have natural generalisation to three space dimensions. 

The structure of this article is as follows: First, in Section~\ref{sec:model}, we derive the governing equations and boundary
conditions for the fluid flow and the gas concentration. Furthermore, we transform the corresponding variational formulation to a fixed
reference domain using the ALE approach. In Section~\ref{sec:disc}, we introduce a finite element discretisation and describe the
pressure stabilisation techniques, especially on moving domains. We present numerical results in Section~\ref{sec:num}, where we
compare the introduced pressure stabilisation techniques and analyse in particular the effect of boundary conditions. We conclude in Section~\ref{sec:conclusion}. 
 
\section{Mathematical model}
\label{sec:model}

\begin{figure}
\centering
 \includegraphics[width=7.5cm]{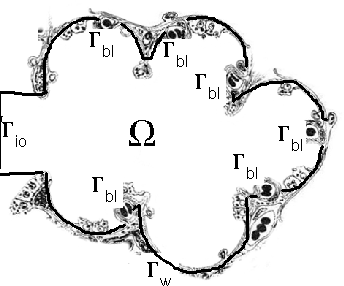}
 \caption{ Scheme of an alveolar sac 
 \label{fig:sac}. The small parts $\Gamma_{bl}$ are 
 the connections to a blood vessel.}
\end{figure}

We are interested in the fluid dynamics and gas exchange in an alveolar sac that is connected to 
the cardiovascular system
(see Figure~\ref{fig:sac}). During inspiration and expiration, the alveolar 
sacs grow and shrink. Thus, the computational
domain $\Omega(t)\subset \mathbb{R}^d$ ($d=2,3$) changes with time. 
We denote the domain velocity by $v^{\text{dom}}$. 

The fluid dynamics in the alveolar sacs are governed by the incompressible Navier-Stokes equations
\begin{align}
      \rho \partial_t v + \rho (v\cdot\nabla) v - \operatorname{div} \sigma
       &=0 \quad \text{in } \Omega(t),&\;\; v&=v^{\text{dom}} \;\;\text{on } \Gamma_{bl}(t) \cup \Gamma_w(t),\label{domeq} \\
       \operatorname{div} v &=0 \;\; \text{in } \Omega(t),  &\sigma n &= 0 \qquad\, \text{on } \Gamma_{io}.
\end{align}
where the Cauchy stress tensor is given by
\begin{align*}
 \sigma = \rho \nu \nabla v - p I,
\end{align*}
$\rho$ and $\nu$ denote the fluid's density and viscosity 
and $v$ and $p$ denote velocity and pressure. At $\Gamma_{io}$, the alveolar sac is connected to the rest of the respiratory system, at
$\Gamma_{bl}(t)$ it is connected to a blood vessel by a semi-permeable membrane through which substances like oxide and carbon dioxide may pass.
The remaining, impermeable boundary of the alveolar sac is denoted by $\Gamma_w(t)$.

As explained in the introduction, alveolar sacs can be found only in the lower part of the 
human lung, between 
the $15^{th}$ and the $23^{rd}$ generation of Weibel's model. 
We consider the density and viscosity constant for all generations $\rho=1.21 \mu g/mm^{3}$
and $\nu\approx14.711 mm^2/s$~\cite{Kvale1975}. The size of the alveolar sacs varies from radial 
distances of $r\approx0.055$ mm in
the $15^{th}$ generation to $r\approx 0.1125$ mm in the $23^{rd}$ generation~\cite{Sznitman2013}.
Considering velocities $V\approx 10^{-2}mm/s$, the Reynolds number $Re=V r/\nu$ ranges
from $Re\approx 3.7\cdot 10^{-5}$ to 
$Re\approx 7.6\cdot 10^{-5}$.

\begin{remark}
 The alveolar sacs move mainly due to a movement of the diaphragm and breasts, which causes a
 low pressure in the surroundings of the alveolar sacs. In the present work, we do not consider the exterior space,
 as models and data for this region are unknown. Instead we consider the movement of the outer boundary of the alveolar sac 
 as known (as it can be observed by experiments, at least at a larger scale) and model in- and outflow on 
 $\Gamma_{io}$ by a so-called do-nothing condition. An alternative approach would be to assume
 the in- and outflow velocity on $\Gamma_{io}$ as known (from experiments) and include the domain movement as a 
 further variable. As the movement of the alveolar sacs is mainly caused by exterior forces, 
 however, we consider the first
 approach as the more natural one.
\end{remark}

\noindent For the concentration of CO$_2$, we assume a convection-diffusion equation 
\begin{align*}
 \partial_t c + v\cdot \nabla c - D \Delta c = 0 \quad\text{in } \; \Omega(t).
\end{align*}
Due to the high permeability of the membrane to the cardiovascular system, the concentrations of CO$_2$ on both sides of the membrane
are often considered to be equal (see e.g. Suresh et al.~\cite{Sureshetal2005}), if one considers only the dissolved portion $c_{bl}$ of CO$_2$ in the blood.
This leads to a Dirichlet condition on $\Gamma_{bl}(t)$ for $c$. The non-diffusivity of the boundary $\Gamma_w$
is described by a homogeneous Neumann condition
\begin{align*}
 c=c_{bl} \quad \text{on } \Gamma_{bl}(t),\qquad
 \partial_n c = 0\quad \text{on } \Gamma_{w}(t).
\end{align*}
Alternatively, a Robin boundary condition $D\partial_n c = \alpha (c_{bl} - c)$ with a high permeability $\alpha$ could be used on $\Gamma_{bl}(t)$.
As values for the constants, we use $D=17\, mm^2 /s$ (\cite{Scherer1988}) 
for the diffusivity and
$c_{bl}=0.06 $ for the CO$_2$ concentrations in the blood~\cite{Hlastala1972}. 
Depending on the different sizes 
of the alveolar sacs in different generations, this corresponds to P\'eclet numbers $Pe= V r / D$ 
ranging from $Pe\approx 3.2\cdot 10^{-5}$ to $Pe\approx 6.6\cdot 10^{-5}$.

On the in- and outflow boundary $\Gamma_{io}$, a classical approach is to impose a homogeneous
Neumann condition in the case of outflow and a Dirichlet condition in the case of inflow, see e.g.~\cite{Scherer1988}.
Denoting by $c_{\text{ext}}$ the concentration of CO$_2$ in the alveolar duct outside, this reads
\begin{align}\label{DirichletNeumann}
 c&=c_{\text{ext}} \quad \text{if } v\cdot n <0,\quad
 \partial_n c = 0 \quad \text{if } v\cdot n \geq0 \quad \text{on } \Gamma_{io}.
\end{align}
For ease of implementation, we include the Dirichlet condition weakly in the 
variational formulation by means of
Nitsche's method~\cite{Nitsche70}. 
Defining the Heaviside function ${\cal H}$ by
\begin{align*}
 {\cal H}(x)= \Bigg\{\begin{split} 1,\quad x&>0,\\
                     0,\quad x&\le 0,
                    \end{split}
\end{align*}
the variational formulation of the complete system of equations reads:
\textit{Find $v\in v^{\text{dom}} + {\cal V}, p \in {\cal L}, c \in c_{bl} + {\cal X}$ such that}
\begin{eqnarray}
\begin{aligned}
(\rho \partial_t v + \rho (v\cdot\nabla) v,\phi) + (\sigma,\nabla \phi)
       &=0 \quad \forall \phi \in {\cal V},\\
       (\operatorname{div} v,\xi) &=0 \quad \forall \xi \in {\cal L},\\
(\partial_t c + v\cdot \nabla c,\psi) + (D \nabla c,\nabla \psi)  \qquad\qquad\qquad&\\
+ \left({\cal H}(-v\cdot n) (\gamma(c-c_{\text{ext}}) -D\partial_n c),\psi \right)_{\Gamma_{io}} &= 0 \quad \forall \psi \in {\cal X}. 
\label{completeSystemEulerian} 
\end{aligned}
\end{eqnarray}
The trial and test spaces are given by
\begin{align}
 {\cal V} &= \left(H^1_0(\Omega(t); \partial \Omega(t) \setminus \Gamma_{io})\right)^2,
 &{\cal L} = L^2(\Omega(t)),\qquad
 {\cal X} &= H^1_0(\Omega(t);\Gamma_{bl}(t)).\label{spaces}
 \end{align}
 We would like to remark that in order to formulate the boundary terms for the convection-diffusion equation, we have to assume
 additional regularity for $c \in {\cal X}$. Here, we have incorporated the Nitsche boundary terms already in the continuous
 formulation in order to simplify the presentation. The well-posedness for this variational formulation can however only be shown on the
 discrete level.
In the discrete setting the Nitsche parameter will be chosen as $\gamma=\gamma_0 h_n^{-1}$, 
where $h_n$ denotes the cell size in normal direction.

The practical problem of this formulation is the choice of the external concentration $c_{\text{ext}}$
in \eqref{DirichletNeumann}. While the average concentration in the interior of the lung could be considered as known
($c_{\text{ext}}\approx 0.04302$~\cite{Hlastala1972}), this value is not necessarily a good approximation for the concentration in 
the alveolar duct right before the alveolar sac. The concentration might be much higher here, as our numerical results will indicate, 
especially immediately after the expiration period, when a large concentration of CO$_2$ has just left through $\Gamma_{\text{io}}$.

Therefore, we will study a different boundary condition in this work as well. The approach is based on the works of Halpern~\cite{Halpern} and 
Loh\'eac~\cite{Loheac} (see also Ehrhardt~\cite{Ehrhardt}), who derived exact ``transparent'' or ``artificial'' boundary conditions for convection-diffusion
equations for the case that an infinite domain is cut at a certain line. Unfortunately, the exact boundary conditions are non-local in time, i.e.$\,$
depending on the concentration $c$ in the whole previous time interval $(0,t)$. 
Therefore, we propose an approximation of these conditions, similarly to the approximations considered by Halpern~\cite{Halpern}:
\begin{align*}
 \partial_t c + D\partial_n c +  [-v\cdot n]_+ (c-c_{\text{ext}}) &= 0,
\end{align*}
where $[f]_+ := \max\{f,0\}$. Note that the transport term appears only in the case of inflow ($v\cdot n< 0$).
Similar conditions have been used by Scherer et al.~\cite{Scherer1988}, who have neglected the transport terms for simplicity.
In the stationary case, the analogous conditions (without the term $\partial_t c$) are known as Danckwerts boundary conditions~\cite{Danckwerts1953}.
The most significant difference compared to the classical boundary conditions \eqref{DirichletNeumann} is the fact that diffusion, which is 
the most important part in this application, is also 
considered during the inflow period, while the classical boundary conditions consider only the exterior concentration $c_{\text{ext}}$ during the 
inflow.

The variational formulation for the concentration $c$ reads
\begin{align}
(\partial_t c + v\cdot \nabla c,\psi) + (D \nabla c,\nabla \psi) 
+ \left(\partial_t c + [-v\cdot n]_+ (c-c_{\text{ext}}),\psi \right)_{\Gamma_{io}} &= 0 \quad \forall \psi \in {\cal X}. 
\label{completeSystemArtf} 
\end{align}
 
The system of equations is complemented with initial conditions for velocity and
concentration $v(\cdot,0) = v^0$, $c(\cdot,0)=c^0$.

\begin{remark}{(Reduced stress tensor)}
Due to the incompressibility condition, we could use the reduced Cauchy 
stress tensor $\sigma$ in equation (\ref{domeq}) instead of the 
full symmetric Cauchy stress tensor
\begin{align}\label{fullSymStress}
 \sigma^{\text{sym}} = \rho \nu (\nabla v + \nabla v^T) - p I.
\end{align}
Using the reduced stress tensor $\sigma$ the \textit{do-nothing}
condition $\sigma n=\nu \partial_n v - pn = 0$ on $\Gamma_{io}$ is included by 
means of variational principles in 
(\ref{completeSystemEulerian}). This is an appropriate boundary condition for 
a flow field, when the channel is cut, but the flow continues uniformly, as it is the case 
in our application, see Heywood, Rannacher \& Turek~\cite{HeywoodRannacherTurek}.
In contrast to the condition $\sigma^{\text{sym}} n = 0$, the \textit{do-nothing} conserves for 
example Poiseuille flow.
\end{remark}

\begin{remark}{(Well-posedness)}
It is well-known that the non-stationary Navier-Stokes equations with homogeneous Dirichlet boundary conditions 
on a fixed domain $\Omega$
possess a global unique solution in 
2 dimensions, while in 3 space dimensions such a result has only been shown for small initial 
data or locally in time (see e.g.~\cite{Temam}). These results have 
been extended in~\cite{HeywoodRannacherTurek} to include \textit{do-nothing} boundary conditions, where global (in time) existence
is proven for small initial data and local existence for larger initial data, in both two and three
space dimensions. Note that this is in contrast to the stationary 
case, where the Navier-Stokes system might not be well-posed, if flow 
enters through the boundary $\Gamma_{io}$ (see e.g. Arndt, Braack \& Lube~\cite{ArndtBraackLube}).

The well-posedness on moving domains $\Omega(t)$ 
can be shown under similar conditions for a sufficiently regular domain movement
(see e.g.~\cite{DesjardinsEsteban, DziriZolesio, RichterBuch}). For well-posedness 
results for a convection-diffusion
 equation on a moving domain with standard boundary conditions, we refer to~\cite{DziriZolesio, Cortez2013}). As discussed previously,
 the well-posedness of the Nitsche formulation is typically shown on the discrete level only. Using the 
 boundary terms \eqref{completeSystemArtf}, the well-posedness can be shown using the techniques from Halpern~\cite{Halpern}. 
The combination of both sub-systems does not pose any additional difficulties, as the coupling between 
them is unidirectional, i.e.$\,$there is no 
feedback from the concentration $c$ to the flow variables $v$ and $p$.
\end{remark}

\subsection{Arbitrary Lagrangian Eulerian method}

To cope with the domain movement, we use the \textit{Arbitrary Lagrangian Eulerian method} 
(ALE) \cite{Donea1977, Belytschko80, HuLiZi81, DoneaSurvey}.
The basic idea of the ALE method is to formulate an equivalent system of equations 
on a fixed reference domain $\hat\Omega$.
Therefore, we introduce a bijective map
\begin{align*}
 \hat T: \hat\Omega \to \Omega(t)
\end{align*}
from the reference domain to the current domain. We assume for simplicity, that 
the in- and outflow boundary $\Gamma_{io}$ remains fixed under $\Tr$.

Quantities $\hat f$ defined in 
the reference systems are related to quantities 
$f$ in the current system by the relation
\begin{align*}
 \hat{f}(\hat{x},t) = f(x,t), \quad x=\hat{T}(\hat{x},t). \label{ALEEuler}
\end{align*}
By $\Fr=\nablar\Tr$ we denote the gradient of $\Tr$ and by $\Jr = \text{det}\, \Fr$ its determinant.
Given sufficient regularity of the mapping $\hat T$ and its inverse $\hat{T}^{-1}$, the system (\ref{completeSystemEulerian}) 
is equivalent to the following system formulated on the fixed domain $\hat\Omega$: \\
\textit{Find $\vr\in \vr^{\text{dom}} + {\hat{\cal V}}, \pr \in \hat{\cal L}, \hat{c} \in \hat{c}_{bl} + \hat{\cal X}$ such that}
\begin{align*}
\rho(\Jr(\Dt\vr+\nablar\vr \Fr^{-1}(\vr-\Dt\Tr),\phir)_{\Fd}+(\Jr\sigmar\Fr^{-T},\nablar\phir)_{\Fd}&=0 \quad\forall \hat\phi \in \hat{\cal V},\\
(\hat{\rm div}(\Jr\Fr^{-1}\vr),\xhir)_{\Fd}&=0 \quad\forall \xhir \in \hat{\cal L},\\
(\Jr(\Dt \hat{c} +(\vr-\Dt\Tr)^T \Fr^{-T} \nablar \hat{c},\hat\phi)_{\Fd}+(D\Jr\Fr^{-T}\nablar \hat{c},\Fr^{-T}\nablar \hat\psi)_{\Fd} 
 \qquad\\
+ \left( {\cal H}(-\vr\cdot \hat{F}^{-T}\hat{n}) (\gamma (\hat{c} -\hat{c}_{\text{ext}}) 
-D \frac{\Fr^{-T}\hat{n}}{\|\Fr^{-T}\hat{n}\|}\cdot \hat{F}^{-T} \hat{\nabla} \hat{c}),\hat\psi\right)_{\hat{\Gamma}_{io}} &= 0 \quad \forall \hat\psi \in \hat{\cal X}.\\
\end{align*}
For details on the transformation rules, see e.g.~\cite{DunnePhD, RichterBuch}. The trial and test spaces are defined analogously to (\ref{spaces}) by 
replacing the domains $\Omega(t)$ by $\hat{\Omega}$ and the boundary part $\Gamma_{bl}(t)$ by $\hat{\Gamma}_{bl}$.
The Cauchy stress tensor is given by
\begin{equation*}
\sigmar:=-\pr I +\rho \nu(\nablar \vr \Fr^{-1}).
\end{equation*}

The variational formulation for the gas concentration with artificial boundary conditions\eqref{completeSystemArtf} transforms to
\begin{align*}
(\Jr(\Dt \hat{c} +(\vr-\Dt\Tr)^T \Fr^{-T} \nablar \hat{c},\hat\phi)_{\Fd}+(D\Jr\Fr^{-T}\nablar \hat{c},\Fr^{-T}\nablar \hat\psi)_{\Fd} 
 \qquad\\
+\left(\partial_t \hat{c} + [-\vr\cdot \Fr^{-T}\hat{n}]_+ \hat{c}-c_{\text{ext}},\hat\psi \right)_{\hat{\Gamma}_{io}} &= 0 \quad \forall \hat\psi \in \hat{\cal X}.
\end{align*}

\section{Discretisation and stabilisation}
\label{sec:disc}

For discretisation, we use $Q_1$ equal-order elements for velocity, pressure and gas concentration
on a mesh $\hat{\Omega}_h$. For simplicity, the discretisation and stabilisation 
is presented here for the two-dimensional case. The generalisations to three space dimensions are however 
straight-forward.
As the discrete \textit{inf-sup} condition is violated, we add stability terms to
guarantee the well-posedness of the fluid equations. The incompressibility condition is modified to
\begin{align}\label{modIncompr}
(\hat{\rm div}(\Jr\Fr^{-1}\hat{v}_h),\hat{\xi}_h)_{\hat{\Omega}_h} + S_h(\hat{p}_h, \hat{\xi}_h) &=0 \quad\forall \hat{\xi}_h \in \hat{\cal L}_h.
\end{align}

Here we study two frequently used and related approaches for
pressure stabilisation.
As first approach, we use the \textit{Local Projection Stabilisation} method (LPS) by Becker \& Braack
\cite{BeckerBraack2001}. Secondly, we study an \textit{interior penalty} technique developed by Burman \& Hansbo 
\cite{BurmanHansbo2006edge, BurmanFernandezHansbo2006}. 
The first approach has been studied in detail for flow problems on anisotropic domains~\cite{BraackRichter2006}. 
The correct form of the stabilisation terms within the
ALE formulation is novel, however. For the derivation, we consider the stabilisation terms on the moving domains 
$\Omega(t)$ first and transform the terms to the reference frame taking care of the anisotropies.

We assume that the  mesh $\Omega_h$ has a patch-hierarchy in the sense, that always four
adjacent quads
arise from refinement of one common patch element. We denote the mesh of patch elements by
$\Omega_{2h}$.

\subsection{Local Projection Stabilisation}

On a Cartesian mesh, the \textit{Local Projection Stabilisation} method adds the stabilisation term
\[
 S_{\text{LPS}}(p_h,\xi_h) = \alpha_{\text{LPS}} \sum_{P\in\Omega_{2h}} (h_x^2 \partial_x \kappa_h p_h, \partial_x \kappa_h \xi_h)_{P} 
 + (h_y^2 \partial_y \kappa_h p_h, \partial_y \kappa_h \xi_h)_{P}
\]
to the divergence equation. Here, we use the projection operator $\kappa_h = \text{id} - i_{2h}$ and $i_{2h}$ 
denotes the linear interpolation from $\Omega_h$ to $\Omega_{2h}$ (cf.~\cite{BraackRichter2006}). For the case
of more general meshes, the stabilisation term might be defined in terms of two coordinate directions $\eta_1,\eta_2$
\begin{align}
 S_{\text{LPS}}^c(p_h,\xi_h) = \alpha_{\text{LPS}} \sum_{P\in\Omega_{2h}} \sum_{i=1}^2 (h_i^2 \partial_{\eta_i} \kappa_h p_h, \partial_{\eta_i} \kappa_h \xi_h)_{P},\label{SLPS_current}
 \end{align}
where $\partial_{\eta_i}=\eta_i\cdot \nabla$ denotes the directional derivative and $h_i$ is the cell size in direction $\eta_i$ (see Figure~\ref{fig:transform}).
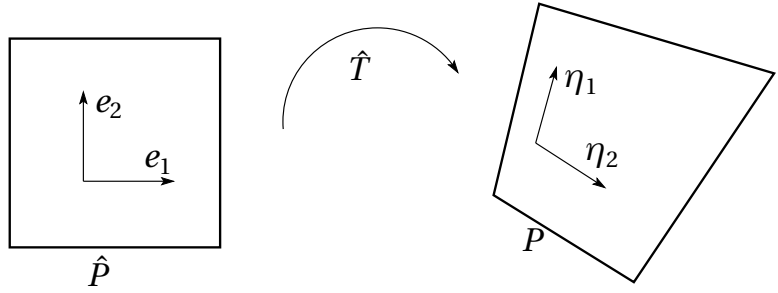
\begin{figure}[bt]
\centering
 \resizebox*{0.7\textwidth}{!}{
 \begin{picture}(0,0)%
\includegraphics{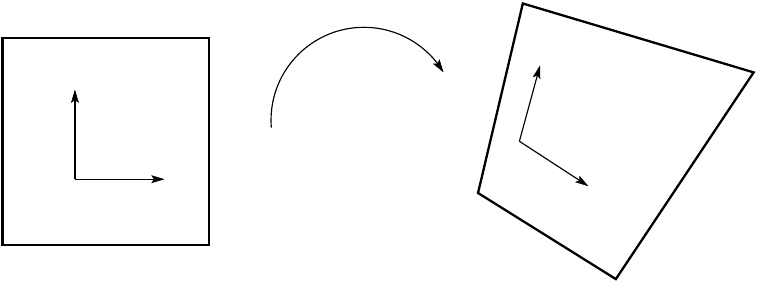}%
\end{picture}%
\setlength{\unitlength}{1450sp}%
\begingroup\makeatletter\ifx\SetFigFont\undefined%
\gdef\SetFigFont#1#2{%
  \fontsize{#1}{#2pt}%
  \selectfont}%
\fi\endgroup%
\begin{picture}(9876,3751)(868,-6929)
\put(2611,-5326){\makebox(0,0)[lb]{\smash{{\SetFigFont{5}{6.0}{\color[rgb]{0,0,0}$e_1$}%
}}}}
\put(1981,-4606){\makebox(0,0)[lb]{\smash{{\SetFigFont{5}{6.0}{\color[rgb]{0,0,0}$e_2$}%
}}}}
\put(8011,-4291){\makebox(0,0)[lb]{\smash{{\SetFigFont{5}{6.0}{\color[rgb]{0,0,0}$\eta_1$}%
}}}}
\put(8281,-5236){\makebox(0,0)[lb]{\smash{{\SetFigFont{5}{6.0}{\color[rgb]{0,0,0}$\eta_2$}%
}}}}
\put(5221,-4201){\makebox(0,0)[lb]{\smash{{\SetFigFont{5}{6.0}{\color[rgb]{0,0,0}$\hat{T}$}%
}}}}
\put(1891,-6856){\makebox(0,0)[lb]{\smash{{\SetFigFont{5}{6.0}{\color[rgb]{0,0,0}$\hat{P}$}%
}}}}
\put(7471,-6406){\makebox(0,0)[lb]{\smash{{\SetFigFont{5}{6.0}{\color[rgb]{0,0,0}$P$}%
}}}}
\end{picture}%
 }
 \caption{\label{fig:transform} ALE map of a cell $P$ and the local coordinate direction $\eta_1,\eta_2$.}
\end{figure}

In order to simplify the presentation, we assume that the reference grid consists of Cartesian quadrilaterals $K$ with
edge sizes $\hat{h}_1$ and $\hat{h}_2$. The cells might be arbitrarily anisotropic, however. We remark that this assumption is not 
necessary in general, but serves to simplify the derivation of the stabilisation term here. Furthermore, note that
the corresponding moved mesh in the current configuration, is not Cartesian in general. 

We use the formulation (\ref{SLPS_current}) on the current domain to derive the ALE stabilisation on the reference domain. By the ALE map the scaled unit vectors are mapped to
\[
 \tilde\eta_1 := \hat{F} \begin{pmatrix} \hat{h}_1 \\ 0 \end{pmatrix}, \quad \tilde\eta_2 := \hat{F} \begin{pmatrix}  0  \\ \hat{h}_2 \end{pmatrix},
 \quad \eta_i = \frac{\tilde\eta_i}{\|\tilde\eta_i\|} = \frac{\tilde\eta_i}{h_i}. 
\]
where $\hat{h}_i$ is the length in horizontal or vertical direction of the Cartesian grid on the reference domain. Note that $F$ and thus $\eta_i$ are in general not constant
within a cell and that the resulting vectors $\eta_1,\eta_2$ are only orthogonal in the case that the ALE map is a translation or a rotation.
As long as the ALE map does not degenerate, the two vectors will
be linearly independent, however, such that stability in any coordinate direction is ensured. This observation holds true in the case of
arbitrary stretching or compression of cells. 

\noindent With these definitions, it holds that
\[
 h_1 \partial_{\eta_1} p_h = h_1\eta_1 \cdot \nabla p_h = \hat{F} \begin{pmatrix} \hat{h}_1 \\ 0 \end{pmatrix} \cdot \hat{F}^{-T} \hat{\nabla} \hat{p}_h = \hat{h}_1 \hat\partial_1 \hat{p}_h.
\]
By the same argumentation, we have $h_2 \partial_{\eta_2} p_h = \hat{h}_2 \hat\partial_2 \hat{p}_h$. Altogether, this yields
\begin{align}\label{stabLPS}
 S_{\text{LPS}}(p_h,\xi_h) &= \alpha_{\text{LPS}} \sum_{P\in\Omega_{2h}} \sum_{i=1}^2 (h_i^2 \partial_{\eta_i} \kappa_h p_h, \partial_{\eta_i} \kappa_h \xi_h)_{P} \\
 &= \alpha_{\text{LPS}}  \sum_{P\in\Omega_{2h}} \sum_{i=1}^2\left(\hat{J} \hat{h}_i^2 \hat{\partial}_i \kappa_h \hat{p}_h, \hat{\partial}_i \kappa_h \hat{\xi}_h\right)_{\hat{P}},
\end{align}
where the determinant $\hat{J}$ appears due to integral transformation. As this stabilisation is equivalent to the LPS stabilisation technique proposed by 
Braack \& Richter~\cite{BraackRichter2006} for anisotropic grids, stability
and convergence estimates can be shown analogously to their result on the  current system $\Omega(t)$.

\subsection{Edge stabilisation}

As a second possibility, we study an \textit{interior penalty} technique~\cite{BurmanFernandezHansbo2006}. Let ${\cal E}_h$ be the set of interior edges of the triangulation $\Omega_h$.
For anisotropic grids, the stabilisation term is usually defined by
\begin{align}\label{StandardIP}
 S_e^c(p_h,\xi_h) = \alpha_{ip}\sum_{e \in {\cal E}_h} \int_e h_n^3 [\partial_n p_h]_e\, [\partial_n \xi_h]_e \, do.
\end{align}
Here, $\partial_n$ is the normal derivative and the brackets $[\cdot]_e$ stand for the jump of the respective function 
over the edge $e$. Furthermore, $h_n$ denotes the size of a cell in the direction orthogonal to the edge $e$. It is sufficient
to stabilise in normal direction of the edges, as the tangential derivatives vanish by the continuity of the space $V_h$.
If the mesh has a patch hierarchy, it is sufficient to sum over all interior patch edges.

Here, we will use a slightly different stabilisation term. In addition to the current normal $n$, we use the ALE-transform of the normal vector of
the reference system
\[
 n_F = \frac{\hat{F} \hat{n}}{\|\hat{F}\hat{n}\|}, \quad h_{n_F} = \hat{h}_n {\|\hat{F}\hat{n}\|}. 
\]
\begin{remark}{(Mapped normal vector)}
 Note that for the normal vector in the current configuration (i.e. the vector that is normal to the mapped tangential vector), it holds that
 \begin{align*}
  n=\frac{\hat{F}^{-T} \hat{n}}{\|\hat{F}^{-T} \hat{n}\|}.
 \end{align*}
This vector is in general only equal to the mapped normal vector $n_F$, if the mapping is a translation or a rotation. Otherwise the mapped normal
$n_F$ is not orthogonal to the mapped tangential vector.
\end{remark}

\noindent We define the stabilisation term by 
\begin{align}\label{defStab}
 S_e(p_h,\xi_h) = \alpha_{ip} \sum_{e \in {\cal E}_h} \int_e h_n h_{n_F}^2 [\partial_{n_F} p_h]_e\, [\partial_{n_F} \xi_h]_e \, do.
 \end{align}
Now, we have again the relation
\[
 h_{n_F} \partial_{n_F} p_h = h_{n_F} n_F \cdot \nabla p_h = \hat{h}_n \hat{F} \hat{n} \cdot \hat{F}^{-T} \hat{\nabla} \hat{p}_h = \hat{h}_n \hat{\partial}_{\hat{n}} \hat{p}_h.
\]
The volume element appearing from integral transformation between an integral over an edge $e$ and an integral 
over the corresponding edge $\hat{e}$ on the reference element is given by 
\[
\hat{J}_e = \hat{J}/ \|\hat{F}^{-T}\hat{n}\| = \hat{J} \frac{\hat{h}_n}{h_n}.
\]
Thus, by integral transformation, we have the equality
\begin{align}
\begin{split}\label{stabEdge}
 S_e(p_h,\xi_h) &= \alpha_{ip} \sum_{e \in {\cal E}_h} 
 \int_e h_n h_{n_F}^2 [\partial_{n_F} p_h]_e\, [\partial_{n_F} \xi_h]_e \, do \\
 &= \alpha_{ip} \sum_{\hat{e} \in \hat{\cal E}_h} \int_{\hat e} 
 \hat{J} \hat{h}_n^3 [\hat\partial_{\hat{n}} \hat{p}_h]_{\hat{e}}\, [\hat{\partial}_{\hat{n}} \hat{\xi}_h]_{\hat{e}} \, d\hat{o}. 
 \end{split}
\end{align}

\begin{remark}{(Choice of normal vectors)}
 In (\ref{defStab}), the factor $h_n$ is needed to compensate the determinant $\hat{J}$ coming from the integral transformation. 
 The factors $h_{n_F}$ are needed to compensate the derivatives $\partial_{n_F}$ that may get large when the grid becomes highly 
 anisotropic. Alternatively, we could have used (\ref{StandardIP}) as starting point. We would obtain the stabilisation term
 \begin{align}\label{edgestab_altern}
  \tilde{S}_e(p_h,\xi_h) = \alpha_{ip} \sum_{\hat{e} \in \hat{\cal E}_h} 
  \int_{\hat e} \hat{J} \hat{h}_n \hat{h}_{\tilde{n}}^2 
  [\hat\partial_{\tilde{n}} \hat{p}_h]_{\hat{e}}\, [\hat{\partial}_{\tilde{n}} \hat{\xi}_h]_{\hat{e}} \, d\hat{o}, 
 \end{align}
where $\tilde{n}=\frac{\hat{F}^{-1} n}{\|\hat{F}^{-1} n\|}$ is the pre-image of the normal vector $n$.

As long as the ALE map does not degenerate, however, $n_F$ and the tangential vector $\tau$ are linearly independent as well, such that the
stabilisations (\ref{defStab})=(\ref{stabEdge}) stabilise in both coordinate directions. As (\ref{stabEdge}) is much easier to evaluate compared to (\ref{edgestab_altern}),
we prefer to use this stabilisation term.
\end{remark}

\subsection{Equivalence}

If the mesh has a patch-hierarchy and if we use only interior edges of patches for the edge stabilisation,
both stabilisation techniques are equivalent, in the sense that both stabilisation terms are in the discrete spaces upper- and
lower-bounded by each other:
 
\begin{lemma}
Assume that there exists a regular triangulation $\hat{\Omega}_{2h}$, whose elements (called 
``patches'') are the union
of four neighbouring cells of the triangulation $\hat{\Omega}_h$. If we consider only interior patch edges
in the definition of the interior penalty stabilisation $S_e$ (\ref{stabEdge}), 
there is a constant C such that
\begin{align}\label{equiv}
 \frac{1}{C} S_e(p_h,p_h) \leq S_{\text{LPS}}(p_h,p_h) \leq C S_e(p_h,p_h). 
\end{align}
\end{lemma}
\begin{proof}
To see this, we denote the patch-wise contributions in the stabilisation terms by
$S_{e,P}$ and $S_{LPS,P}$ and transform from a patch $\hat{P}$ in ALE coordinates 
to a unit patch $\check{P}$ by means of the inverse of the bi-linear, bijective map 
$\Theta_{\hat{P}}: \check{P} \to \hat{P}$. The transformed terms on the unit patch are denoted by
\begin{align*}
 \check{S}_{e,\check{P}}(\check p_h, \check \xi_h) 
 &:=\alpha_{ip} \sum_{\check{e} \subset \text{int}(\check{P})} \int_{\check e} 
 [\check\partial_{\check{n}} \check{p}_h]_{\check{e}}\, 
 [\check{\partial}_{\check{n}} \check{\xi}_h]_{\check{e}} \, d\check{o}, 
 \\
 \check{S}_{LPS, \check{P}}(\check p_h, \check \xi_h) &:= \alpha_{\text{LPS}} 
 \sum_{i=1}^2\left(\check{\partial}_i \kappa_h \check{p}_h, 
 \check{\partial}_i \kappa_h \check{\xi}_h\right)_{\check{P}},
\end{align*}
where functions on the unit patch are defined by $\check p_h(\check{x}) := \hat{p}_h(\Theta_{\hat{P}} (\check{x}))$ and 
the direction of the derivatives $\check{\partial}_i$ 
and $\check{\partial}_{\check{n}}$
correspond to the directions of $\hat{\partial}_i$ and $\hat{\partial}_{\hat{n}}$, 
under the application of the map $\Theta_{\hat{P}}^{-1}$.

Then, we show that
\begin{align*}
 \frac{1}{C} S_{e,P}(p_h,p_h) \leq |T| \check{S}_{e,\check{P}}(\check p_h, \check p_h) 
 \leq c |T| \check{S}_{LPS, \check{P}}(\check p_h, \check p_h) \leq C S_{LPS,P}(p_h, p_h).
\end{align*}
The second inequality in \eqref{equiv} follows analogously.

The first and last inequality follow by the usual transformation formulas.
For the estimation on the unit patch $\check{P}$, we show that both stabilisation
terms define norms on the local quotient space $V_h^{loc}\setminus V_{2h}^{loc}$ on $\check{P}$. 
The only norm property,
that is non-obvious is the definiteness, i.e. for $\check{p}_h \in V_h^{loc}\setminus V_{2h}^{loc}$ 
we have to show that
\begin{align}\label{definitness}
 \check{S}_*(\check{p}_h,\check{p}_h) =0 \quad \Rightarrow \quad \check{p}_h =0,
\end{align}
where $\check{S}_* = \check{S}_{\text{LPS},\check{P}}$ or $\check{S}_{e,\check{P}}$.
For the LPS stabilisation, the left statement implies directly $\kappa_h \check{p}_h=$const on each 
of the sub-cells $\check{K}\subset \check{P}$.
As $\kappa_h \check{p}_h(x_i)=0$ by definition in the four outer vertices $x_i$ ($i=0,...,3$), this implies
\eqref{definitness}.

For the ip stabilisation, we proceed in the following way: The left statement in~\eqref{definitness}
implies $[\check\partial_{\check{n}} \check{p}_h] =0$ on the four interior edges of $\check{P}$. 
A function $\check{p}_h \in V_h^{loc}\setminus V_{2h}^{loc}$ is zero in the four outer vertices and is therefore
defined by the five degrees of freedom in the midpoints of edges $a_i (i=0,...,3)$ and the midpoint of 
the patch $m_{\check{P}}$,
see Figure~\ref{fig.equiv}. 

We start by showing that $\check{p}_h(a_0)=0$. Therefore, note that
on the bottom line $\Gamma_0$, we have the conditions 
$\check{p}_h(x_0)
=\check{p}_h(x_1)=0$ and $[\check{\partial}_x \check{p}_h](a_0)=0$. As $\check{p}_h$ is 
linear on both
sides of $a_0$ and continuous across $e_0$, the only possibility to fulfil the jump condition 
is $\check{p}_h(a_0)=0$. With the 
same argumentation, we can show that $\check{p}_h(a_i)=0$ for $i=1,...,3$ and 
$\check{p}_h(m_{\check{P}})=0$.
\end{proof}

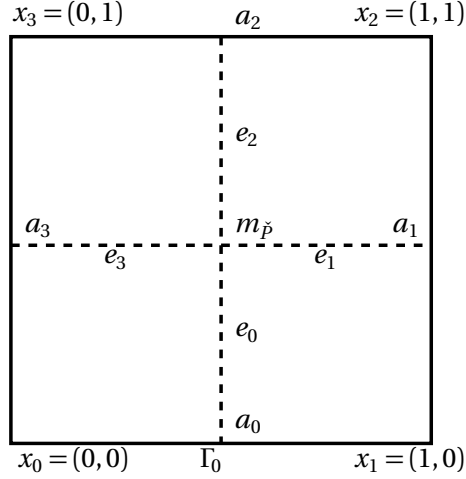
\begin{figure}
\centering
 \resizebox*{0.39\textwidth}{!}{
 \begin{picture}(0,0)%
\includegraphics{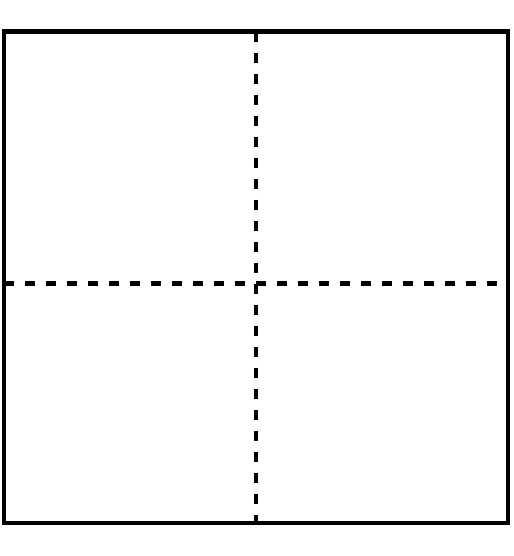}%
\end{picture}%
\setlength{\unitlength}{1326sp}%
\begingroup\makeatletter\ifx\SetFigFont\undefined%
\gdef\SetFigFont#1#2{%
  \fontsize{#1}{#2pt}%
  \selectfont}%
\fi\endgroup%
\begin{picture}(7308,7906)(1747,-8450)
\put(5626,-7711){\makebox(0,0)[lb]{\smash{{\SetFigFont{5}{6.0}{\color[rgb]{0,0,0}$a_0$}%
}}}}
\put(5626,-4336){\makebox(0,0)[lb]{\smash{{\SetFigFont{5}{6.0}{\color[rgb]{0,0,0}$m_{\check{P}}$}%
}}}}
\put(2026,-4336){\makebox(0,0)[lb]{\smash{{\SetFigFont{5}{6.0}{\color[rgb]{0,0,0}$a_3$}%
}}}}
\put(5626,-736){\makebox(0,0)[lb]{\smash{{\SetFigFont{5}{6.0}{\color[rgb]{0,0,0}$a_2$}%
}}}}
\put(5626,-2761){\makebox(0,0)[lb]{\smash{{\SetFigFont{5}{6.0}{\color[rgb]{0,0,0}$e_2$}%
}}}}
\put(5626,-6136){\makebox(0,0)[lb]{\smash{{\SetFigFont{5}{6.0}{\color[rgb]{0,0,0}$e_0$}%
}}}}
\put(3376,-4876){\makebox(0,0)[lb]{\smash{{\SetFigFont{5}{6.0}{\color[rgb]{0,0,0}$e_3$}%
}}}}
\put(6976,-4876){\makebox(0,0)[lb]{\smash{{\SetFigFont{5}{6.0}{\color[rgb]{0,0,0}$e_1$}%
}}}}
\put(8326,-4336){\makebox(0,0)[lb]{\smash{{\SetFigFont{5}{6.0}{\color[rgb]{0,0,0}$a_1$}%
}}}}
\put(7651,-691){\makebox(0,0)[lb]{\smash{{\SetFigFont{5}{6.0}{\color[rgb]{0,0,0}$x_2=(1,1)$}%
}}}}
\put(1801,-691){\makebox(0,0)[lb]{\smash{{\SetFigFont{6}{7.2}{\color[rgb]{0,0,0}$x_3=(0,1)$}%
}}}}
\put(1891,-8386){\makebox(0,0)[lb]{\smash{{\SetFigFont{5}{6.0}{\color[rgb]{0,0,0}$x_0=(0,0)$}%
}}}}
\put(5041,-8386){\makebox(0,0)[lb]{\smash{{\SetFigFont{5}{6.0}{\color[rgb]{0,0,0}$\Gamma_0$}%
}}}}
\put(7651,-8386){\makebox(0,0)[lb]{\smash{{\SetFigFont{5}{6.0}{\color[rgb]{0,0,0}$x_1=(1,0)$}%
}}}}
\end{picture}%
}
 \caption{\label{fig.equiv} Unit patch $\check{P}$. The degrees of freedom of the local space
 $V_h^{loc} \setminus V_{2h}^{loc}$ are located in the midpoints of edges $a_i$ ($i=0,...,3$) and
 the patch midpoint $m_{\check{P}}$.}
\end{figure}

\subsection{Fully discrete system}

For time discretisation, we apply the backward Euler scheme with a uniform time step $\delta t$. 
Given the solutions $\vr_h^{m-1}$ and $\hat{c}_h^{m-1}$ at the previous time-step, one time-step of the fully discrete system
with classical boundary conditions~\eqref{DirichletNeumann} reads:
\textit{Find $\vr_h^m\in \vr^{\text{dom}} + {\hat{\cal V}_h}, \pr_h^m \in \hat{\cal L}_h, 
\hat{c}_h^m \in \hat{c}_{bl} + \hat{\cal X}_h$ such that}
\begin{align*}
\rho(\Jr(\delta t^{-1}(\vr_h^m - \vr_h^{m-1})+\nablar\vr_h^m \Fr^{-1}(\vr_h^m-\Dt\Tr)),\phir)_{\Fd}\;\quad\qquad\qquad&\\
+(\Jr\sigmar(\vr_h^m, \pr_h^m) \Fr^{-T},\nablar\phir)_{\Fd}=0 \;\forall \hat\phi \in \hat{\cal V}_h,&\\
(\hat{\rm div}(\Jr\Fr^{-1}\vr_h^m),\xhir)_{\Fd}=0 \;\forall \xhir \in \hat{\cal L}_h,&\\
(\Jr(\delta t^{-1}(\hat{c}_h^m - \hat{c}_h^{m-1})
+(\vr_h^m-\Dt\Tr)^T \Fr^{-T} \nablar \hat{c}_h^m,\hat\phi)_{\Fd}
+(D\Jr\Fr^{-T}\nablar \hat{c}_h^m,\Fr^{-T}\nablar \hat\psi)_{\Fd}& \\
+\Big( {\cal H}(-\vr_h^{m,T} \hat{F}^{-T}\hat{n}) (\gamma (\hat{c}_h^m -c_{\text{ext}}) 
-D \frac{\Fr^{-T}\hat{n}}{\|\Fr^{-T}\hat{n}\|}\cdot \hat{F}^{-T} \hat{\nabla} \hat{c}_h^m),\hat\psi\Big)_{\Gamma_{io}}
= 0 \; \forall \hat\psi \in \hat{\cal X}_h&.
\end{align*}

Note that the Heaviside function ${\cal H}(-\vr_h^{m,T} \hat{F}^{-T}\hat{n})$ is discretised in a fully implicit way.
This is possible due to the following observations.
To solve the non-linear system of equations, 
we use a Newton-type method.
 As there is no feedback from the concentration $\hat{c}$ to the fluid variables $\hat{v}$ and $\hat{p}$,
the system can be split in each Newton step to solve for the flow variables first and for the 
concentration $\hat{c}$ afterwards. In the numerical simulations conducted for this paper, we have
simply ignored the non-differentiability of ${\cal H}$ at $\vr_h^{m,T} \hat{F}^{-T}\hat{n}=0$,
as this equality was never exactly fulfilled. In very few time step, we observed some issues with 
Newton convergence, as $\vr_h^{m,T} \hat{F}^{-T}\hat{n}$ was changing its sign. In these cases,
a simple damping strategy was enough to recover Newton convergence with only a few extra iterations.

\section{Numerical results}
\label{sec:num}

All our results have been obtained with the finite element 
library \textit{Gascoigne 3d}~\cite{Gascoigne}. To solve the non-linear system of equations
a Newton-type method is used; the resulting linear systems are solved by a direct 
solver (umfpack~\cite{umfpack}).

We start by testing the pressure stabilisation techniques 
for a stationary Stokes problem with a known analytical solution on highly anisotropic domains in Section~\ref{sec.numpress}. Then we study 
the full problem on a representative domain of an alveolar sac in Section~\ref{sec.numsac}.

\subsection{Comparison of the pressure stabilisation techniques}
\label{sec.numpress}

\begin{figure}[bt]
\centering
\includegraphics[width=0.85\textwidth]{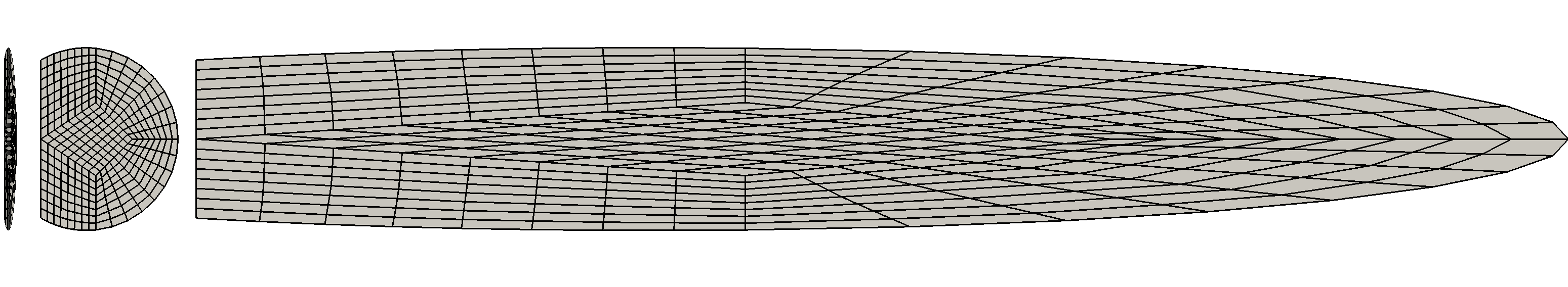}
\caption{ Meshes for $a=0.1$ (left), $a=1$ (middle) and $a=10$ (right), deformed by the mapping $T$. \label{fig:mesh}}
\end{figure}

In this section, we want to study the stabilisation techniques 
on the domains
\begin{align*}
 \Omega_a :=\left\{ x \in \mathbb{R}^2\, \big| \, x_1>0,\, \left( a^{-1} x_1 - 0.5\right)^2 + x_2^2 < 1 \right\}   
\end{align*}
where $a>0$, that are stretched ($a>1$) or compressed ($a<1$) in horizontal direction. $\Omega_a$ is
the image of the domain $\hat{\Omega}:=\Omega_1$ under the mapping
\begin{align*}
x& = T(\hat{x}) = (a \hat{x}_1, \hat{x}_2). 
\end{align*}
This map will also be used in our computations, that will all be done on the reference domain $\hat{\Omega}$, which is 
decomposed by means of regular triangulations ${\cal T}_h$. We will show results for the extreme cases
$a=0.01$ and $a=100$. To illustrate the anisotropies, we show deformed meshes for the less extreme cases $a=0.1$ and $a=10$
in Figure~\ref{fig:mesh}. 
In order to isolate the effect of the pressure stabilisation from further numerical issues, we 
consider the stationary Stokes equations, formulated in ALE coordinates:
\textit{Find $\vr\in \vr^{\text{dom}} + {\hat{\cal V}}, \pr \in \hat{\cal L}$ such that}
\begin{align*}
(\Jr\nablar \vr \Fr^{-1},\nablar\phir \Fr^{-1})_{\Fd} - (\pr,{\rm div} (\Jr\Fr^{-1}\phir))_{\Fd} &= (\Jr \fr, \phir) \quad\forall \hat\phi \in \hat{\cal V},\\
(\hat{\rm div}(\Jr\Fr^{-1}\vr),\xhir)_{\Fd} &=0 \quad\forall \xhir \in \hat{\cal L}.
\end{align*}
The right-hand side $\fr$ will be chosen in such a way that an analytical solution is known. We impose a do-nothing condition on the left boundary $\Gamma_l$ (corresponding to $x_1 = 0$) and
homogeneous Dirichlet conditions for the velocities on the remaining boundary. In order to construct a divergence-free velocity field, we use the potential
\begin{align*}
 \psi(x) = k_a^2(x) \sin^3(a^{-1} x_1), \quad \text{ where } k_a(x) = (a^{-1} x_1 - 0.5)^2 + x_2^2 -1 
\end{align*}
and define
\begin{align*}
 v &:= {\rm rot } \,\psi = (\partial_2 \psi, -\partial_1 \psi),\quad
 p := \partial_{12} \psi.
\end{align*}
By definition of the pressure a do-nothing condition is fulfilled on $\Gamma_l$. The definition of $k_a$ guarantees homogeneous Dirichlet conditions
on $\partial\Omega_a \setminus \Gamma_l$.

We will compare the anisotropic LPS stabilisation $S_{\text{LPS}}$ defined in (\ref{stabLPS}) and the anisotropic edge stabilisation $S_e$ (\ref{stabEdge}) 
to variants of the LPS stabilisation for isotropic domains and meshes
\begin{align}
 S_{\text{LPS}}^{\text{iso}}(\hat{p}_h,\hat{\xi}_h) &:= \alpha_2 \left(\hat{J} \hat{F}^{-T} \hat{\nabla} \hat{p}_h, \hat{F}^{-T} \hat{\nabla} \hat{\xi}_h\right)_{\hat{\Omega}} = \alpha_2 \left(\nabla p_h,\nabla \xi_h\right)_{\Omega(t)}\\
 S_{\text{LPS}}^{\text{simple}}(\hat{p}_h,\hat{\xi}_h) &:= \alpha_3 \left( \hat{\nabla} \hat{p}_h, \hat{\nabla} \hat{\xi}_h\right)_{\hat{\Omega}}.
\end{align}
We use the stabilisation parameters $\alpha_{\text{LPS}}=\alpha_2=\alpha_3=\frac{1}{\nu}$ for the LPS-based stabilisations
and $\alpha_{ip}=\frac{1}{60\nu}$ for the interior-penalty technique. These parameters have been chosen empirically by comparing results in terms of robustness 
and errors for different parameters.

\noindent We compare the effect of the stabilisations by means of the $H^1$-semi-norm error of velocity and the $L^2$-norm errors of pressure and velocity
\begin{align*}
 \|\nabla (v-v_h)\|_{\Omega_{a}} &= \| \Jr^{1/2} \nablar (\vr - \hat{v}_h) \Fr^{-1}\|_{\hat{\Omega}}, \quad
 \|v-v_h\|_{\Omega_{a}} = \| \Jr^{1/2} (\vr - \hat{v}_h)\|_{\hat{\Omega}},\\
  \|p-p_h\|_{\Omega_{a}} &= \| \Jr^{1/2} (\hat{p} - \hat{p}_h)\|_{\hat{\Omega}}.
\end{align*}
Moreover, we introduce the functional
\begin{align*}
 J_{\text{div}}(v_h) &:= \int_{\Omega(t)} ({\rm div} \, v_h)^2 \, dx \, 
 = \,  \int_{\hat{\Omega}}\hat{J}^{-1}  \left(\widehat{\rm div} 
 \big(\hat{J} \hat{F}^{-1} \hat{v}_h\big)\right)^2 \, d\hat{x},
\end{align*}
measuring the error with respect to incompressibility.
Note that $J_{\text{div}}(v)$ vanishes for the continuous solution $v$. 
This is, however, altered by the pressure stabilisation, see (\ref{modIncompr}).

In Table~\ref{tab.press}, we compare the values of the functionals for $a=0.01$, where the domain is highly compressed.
Note that for this parameter the matrix $\hat{F}^{-T}$ is given by
 \begin{align*}
 \hat{F}^{-T} = \begin{pmatrix}
   100 &0 \\ 0 &1
  \end{pmatrix}
 \end{align*}
 and the determinant is $\hat{J}=\frac{1}{100}$. The stabilisation in horizontal direction is thus by a factor 100 larger for both isotropic stabilisation variants 
 $S_{\text{LPS}}^{\text{iso}}$ and $S_{\text{LPS}}^{\text{simple}}$, that tend to over-stabilise
compared to the anisotropic variant $S_{\text{LPS}}$. This can be observed in the functional $J_{\text{div}}(v_h)$, 
where $S_{\text{LPS}}^{\text{iso}}$ and $S_{\text{LPS}}^{\text{simple}}$ yield a slightly larger deviation from zero compared to $S_{\text{LPS}}$. A much smaller value
is obtained for the (consistent) edge stabilisation $S_e$. 

The $L^2$-norm error in the pressure, on the other hand, shows exactly the opposite picture, the smallest errors
being obtained for the isotropic LPS stabilisations. Asymptotically, the behaviour in all the functionals is similar, such that from this test case no clear advantage for 
any of the methods can be deduced. The $L^2$-norm and $H^1$-semi-norm errors in the velocity are almost independent of the pressure stabilisation and show the expected convergence order
for piece-wise bi-linear elements. In the $L^2$-norm of the pressure, super-convergence 
can be observed for all stabilisations. This has frequently been observed 
in literature before, see e.g.~\cite{SuperconvergenceBonitoBurman, SuperconvergenceMassing}.

\begin{table}[bt]
\centering
\begin{tabular}{r|ccc|ccc}
&\multicolumn{3}{c|}{ $J_{\text{div}}(v_h)$} &\multicolumn{3}{c}{ $\|p-p_h\|_{\Omega_a}$ }\\
\#cells & $S_{\text{LPS}}^{\text{iso}/\text{simple}}$ & $S_{\text{LPS}}$ &$S_e$
& $S_{\text{LPS}}^{\text{iso}/\text{simple}}$ & $S_{\text{LPS}}$ &$S_e$\\
\hline
320&$6.63\cdot10^{-3}$  & $6.56\cdot10^{-3}$ &  $2.33\cdot10^{-3}$ & $0.914$ & $0.999$ & $1.217$\\
1280 &$1.12\cdot10^{-3}$ & $1.05\cdot10^{-3}$ &  $3.60\cdot10^{-4}$ & $1.174$ & $1.219$ & $0.902$\\
4920 &$6.59\cdot10^{-5}$  & $6.36\cdot10^{-5}$ &  $5.05\cdot10^{-5}$ & $0.418$ & $0.422$ & $0.406$\\
19680 &$1.09\cdot10^{-5}$ & $1.08\cdot10^{-5}$ &  $1.04\cdot10^{-5}$ & $0.105$ & $0.106$ & $0.108$\\
\hline 
$\alpha_{\text{conv}}$ &2.61 &2.69 &2.70 &1.56 &1.60 &1.29
\end{tabular}
\vspace{0.3cm}

\begin{tabular}{r|ccc|ccc}
&\multicolumn{3}{c|}{ $\|\nabla(v-v_h)\|_{\Omega_a}$} &\multicolumn{3}{c}{ $\|v-v_h\|_{\Omega_a}$ }\\
\#cells & $S_{\text{LPS}}^{\text{iso}/\text{simple}}$ & $S_{\text{LPS}}$ &$S_e$
& $S_{\text{LPS}}^{\text{iso}/\text{simple}}$ & $S_{\text{LPS}}$ &$S_e$\\
\hline
320 &$4.338$   &$4.338$  &$4.337$ & $3.03\cdot10^{-3}$ & $3.03\cdot10^{-3}$ & $3.00\cdot10^{-3}$\\
1280 &$2.202$   &$2.202$  &$2.201$ & $7.81\cdot10^{-4}$ & $7.79\cdot10^{-4}$ & $7.74\cdot10^{-4}$\\
4920 &$1.105$   &$1.105$  &$1.105$ & $2.07\cdot10^{-4}$ & $2.07\cdot10^{-4}$ & $2.05\cdot10^{-4}$\\
19680 &$0.553$  &$0.553$  &$0.553$  & $5.54\cdot10^{-5}$ & $5.54\cdot10^{-5}$ & $5.52\cdot10^{-5}$\\
\hline 
$\alpha_{\text{conv}}$ &0.98 &0.98 &0.98 &1.95 &1.95 &1.95
\end{tabular}

 \caption{\label{tab.press} Comparison of the functional values $J_{\text{div}}(v_h)$, the $L^2$-norms of 
 pressure and the $H^1$-semi- and $L^2$-norm errors of the velocity
 for the four different stabilisation techniques on different meshes for the highly compressed 
 domain $\Omega_{0.01}$. All the functional values for $S_{\text{LPS}}^{\text{simple}}$ and $S_{\text{LPS}}^{\text{iso}}$
 are identical in the first three digits and are therefore combined in one column. The convergence order is estimated by means of a least-squares-fit
 of the function $f(h) =ch^{\alpha_{\text{conv}}}$ against $c$ and $\alpha_{\text{conv}}$.}
\end{table}

Next, we study in Table~\ref{tab.press100} the case $a=100$, i.e.$\,$a 
highly stretched domain. In this case, the matrix $\hat{F}^{-T}$ is given by
 \begin{align*}
 \hat{F}^{-T} = \begin{pmatrix}
   0.01 &0 \\ 0 &1
  \end{pmatrix}
 \end{align*}
 and the determinant is $\hat{J}=100$. The stabilisation in horizontal direction is thus by a factor 100 smaller for both isotropic variants
 $S_{\text{LPS}}^{\text{iso}}$ and 
$S_{\text{LPS}}^{\text{simple}}$, which means
that these might not stabilise enough. 

Indeed, the discrete system was ``numerically singular'' on the finest grid with 19'680 cells
for the isotropic stabilisations $S_{\text{LPS}}^{\text{iso}}$ and $S_{\text{LPS}}^{\text{simple}}$, 
i.e.$\,$the linear (direct) solver was not able 
to solve the linear system. Varying the stabilisation parameters $\alpha$, we found that
we would have to choose the larger stabilisation parameters 
$\alpha\geq 10$ and $\alpha\geq 15$ for $S_{\text{LPS}}^{\text{simple}}$
and $S_{\text{LPS}}^{\text{iso}}$, respectively. However, with this choice of parameters 
the same issue arises on the next-finer mesh. For the 
anisotropic stabilisations $S_{\text{LPS}}$ and $S_e$, on the other hand, no issues regarding 
the solution of the linear systems were observed.

While the convergence of the functional values $J_{\text{div}}(v_h)$ is comparable for all stabilisations (in the cases where the system could be solved),
the $L^2$-norm of the pressure shows a clear advantage for the anisotropic LPS stabilisation 
compared to the isotropic variants. The velocity norm and semi-norm errors are
again almost independent of the pressure stabilisation and converge as expected. Note that the reason 
for the larger absolute values of the functionals compared to $a=0.01$, is that the domain 
$\Omega_a$ is by a factor of $10^4$ longer.

\begin{table}[bt]
\centering
\begin{tabular}{r|cccc}
&\multicolumn{4}{c}{ $J_{\text{div}}(v_h)$}\\
\#cells & $S_{\text{LPS}}^{\text{iso}}$ &$S_{\text{LPS}}^{\text{simple}}$ & $S_{\text{LPS}}$ &$S_e$\\
\hline
320&$5.985$ & $6.079$ & $6.023$ &  $6.002$\\
1280 &$1.548$ & $1.569$ & $1.547$ &  $1.554$\\
4920 &$0.391$ & $0.394$ & $0.391$ &  $0.392$\\
19680 &$-$ & $-$ & $0.098$ &  $0.098$\\
\hline 
 $\alpha_{\text{conv}}$ &1.95 &1.96 &1.96 &1.95
\end{tabular}
\vspace{0.3cm}

 \begin{tabular}{r|cccc}
 &\multicolumn{4}{c}{ $\|p-p_h\|_{\Omega_a}$ }\\
\#cells & $S_{\text{LPS}}^{\text{iso}}$ &$S_{\text{LPS}}^{\text{simple}}$ & $S_{\text{LPS}}$ &$S_e$\\
\hline
320 &$9.67\cdot 10^{-1}$ & $1.71\textcolor{white}{\cdot 10^{-11} }$ & $8.88\cdot 10^{-1}$ & $1.02\textcolor{white}{\cdot 10^{-11} }$\\
1280 &$3.91\cdot 10^{-1}$ & $7.31\cdot 10^{-1}$ & $2.55\cdot 10^{-1}$ & $4.21\cdot 10^{-1}$\\
4920 &$1.70\cdot 10^{-1}$ & $3.08\cdot 10^{-1}$ & $1.13\cdot 10^{-1}$ & $1.98\cdot 10^{-1}$\\
19680 &$-$ & $-$ & $4.74\cdot 10^{-2}$ & $6.82\cdot 10^{-2}$\\
\hline 
$\alpha_{\text{conv}}$ &1.28 &1.23 &1.67 &1.24
\end{tabular}
 \vspace{0.3cm}

\begin{tabular}{r|cc|cc}
&\multicolumn{2}{c|}{ $\|\nabla(v-v_h)\|_{\Omega_a}$}  &\multicolumn{2}{c}{ $\|v-v_h\|_{\Omega_a}$ }\\
\#cells & $S_{\text{LPS}}^{\text{iso}/\text{simple}}$ & $S_{\text{LPS}}=S_e$  
& $S_{\text{LPS}}^{\text{iso}/\text{simple}}$ & $S_{\text{LPS}}=S_e$\\
\hline
320 &$2.13\cdot10^2$   &$2.13\cdot 10^2$ & $1.39\cdot10^1\;\;\,$ & $1.39\cdot10^1\;\;\,$\\
1280 &$1.08\cdot10^2$   &$1.08\cdot 10^2$ & $3.56\textcolor{white}{\cdot 10^{-11} }$ & $3.56\textcolor{white}{\cdot 10^{-11} }$\\
4920 &$5.46\cdot10^1$  &$5.46\cdot 10^1$ & $9.00\cdot10^{-1}$ & $9.00\cdot10^{-1}$\\
19680 &$-$ &$2.73\cdot 10^1$ & $-$ & $2.26\cdot10^{-1}$\\
\hline 
$\alpha_{\text{conv}}$ &0.98 &0.98 &1.96 &1.96
\end{tabular}

 \caption{\label{tab.press100} Comparison of functional values  
 for the highly stretched 
 domain $\Omega_{100}$. For the isotropic LPS variants, the direct solver could not solve the linear system on the finest mesh. 
The values of the velocity norms for the different stabilisations
 were identical up to the third digit in all cases the solver converged. The convergence order has been estimated as in Table~\ref{tab.press}.}
\end{table}

\subsection{Examples on an alveolar sac geometry}
\label{sec.numsac}

In this section we consider an alveolar sac geometry consisting of 5 alveoli with 6 
connections $\Gamma_{bl}(t)$ to the 
cardiovascular system, inspired by Figure~\ref{fig:sac}. In average, 
an alveolar sac grows and shrinks by about 9 \% compared to the intermediate state~\cite{Sznitman2013}. 
We use the domain map 
\begin{align}\label{dommap}
x(t)& = T_1(\hat{x}) = \begin{cases}
                       \;\hat{x}(1-a \cos(0.4\pi t)), \quad &\hat{x}\geq 0,\\
                       \;\hat{x},  \quad &\hat{x}\leq 0,
                      \end{cases}\qquad
y(t)=\hat{y},
\end{align}
with $a=0.09$ describing a sinusoidal movement in horizontal direction. Accordingly,
the domain velocity is given by $v^{\text{dom}}=\partial_t \hat{T} =0.4 a \pi \sin(0.4\pi t) [\hat{x}]_+$.

We apply the discretisation techniques and parameters introduced above in combination with 
anisotropic LPS stabilisation. For time discretisation, we use the backward Euler 
time-stepping scheme with time step $\delta t=0.05$. The Nitsche parameter is chosen $\gamma_0=10$
and as initial values we use $v^0=v^{\text{dom}}(0)$ and $c^0 := c_{\text{ext}} = 0.04302$. 

\paragraph{Influence of the boundary conditions}

Some first results using the classical boundary conditions \eqref{DirichletNeumann}
in combination with Nitsche's method on $\Gamma_{io}$ 
are shown in Figure~\ref{fig.shortduct}. Here, we have added a short alveolar duct of length $l=0.6$mm on the left-hand side and chose
$c_{\text{ext}}=0.04302$ as Dirichlet value, which is an approximation for the average CO$_2$ concentration in the interior of the lung~\cite{Hlastala1972}.
During inspiration (top left)
the CO$_2$ concentration takes values ranging from around 0.056 to 0.06 in the alveolar sac. This changes quickly, when the boundary condition for 
$c$ changes to a Neumann condition at time $t=2.5$s (top right). At time $t=2.6$s, the concentration $c$ in the alveolar sac is already almost uniformly $c\approx 0.06$, which is
the Dirichlet value prescribed at the channels $\Gamma_{bl}(t)$, due to the large diffusion coefficient $D$.

\begin{figure}[bt]
\centering
\includegraphics[width=0.45\textwidth]{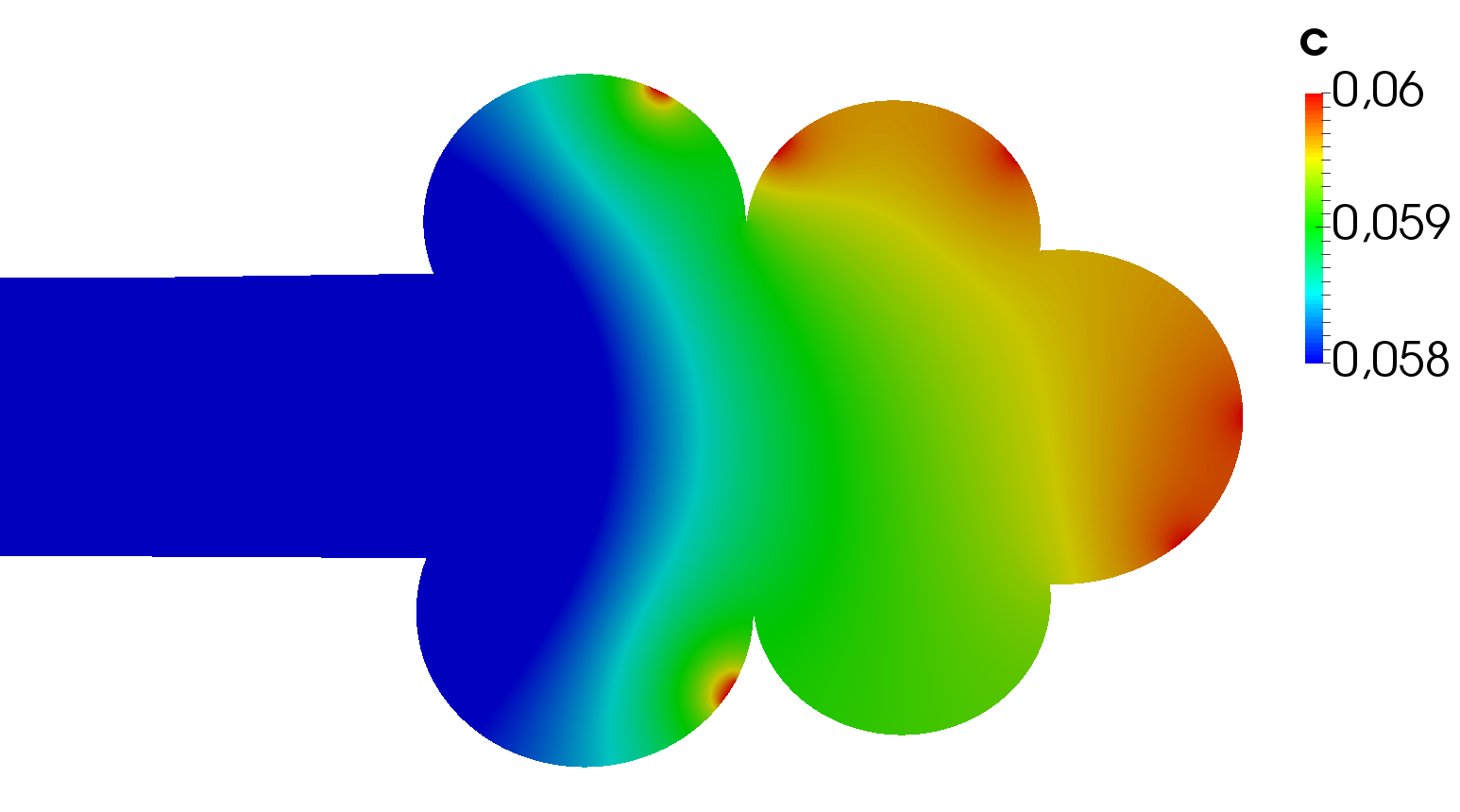}
\includegraphics[width=0.45\textwidth]{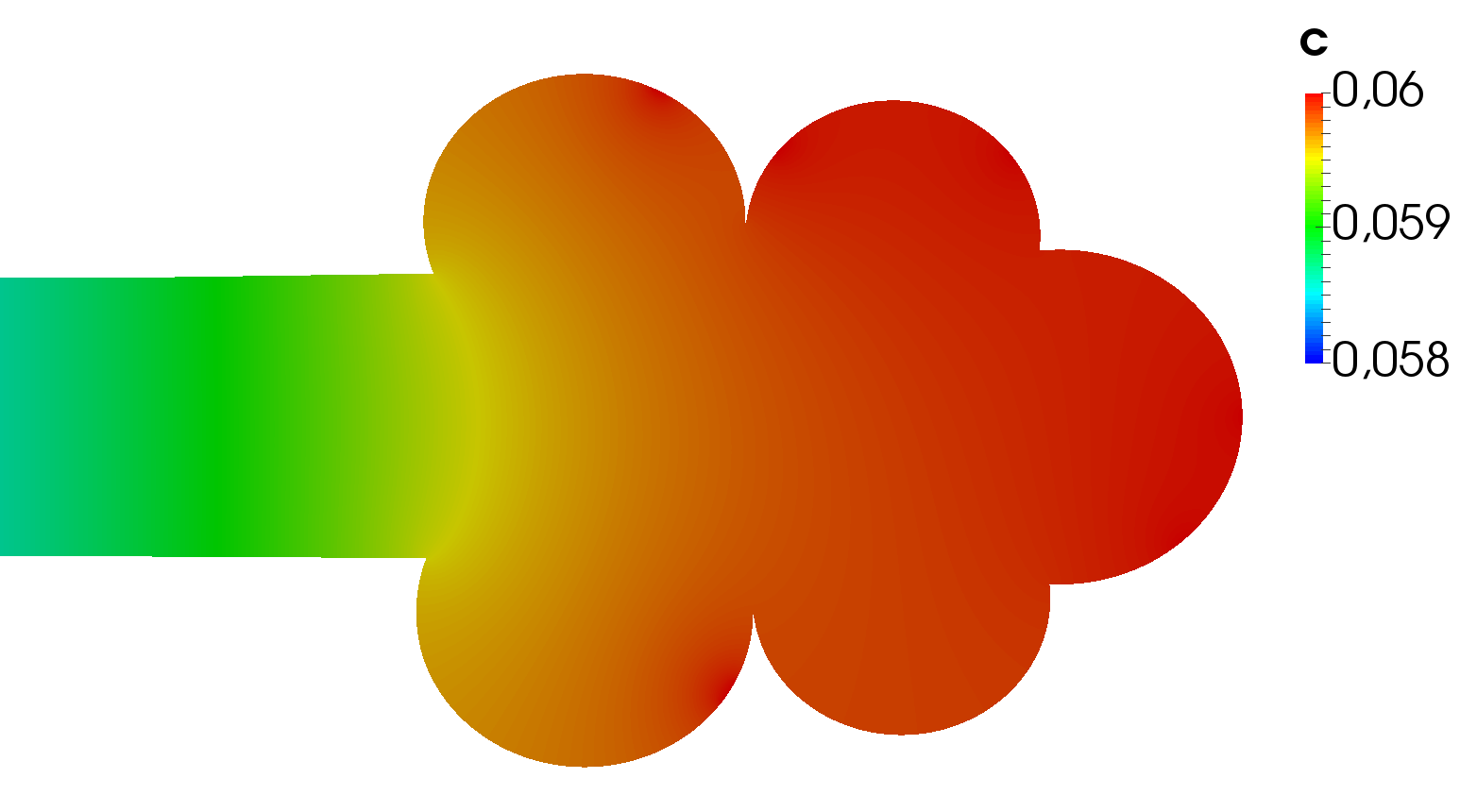}
\includegraphics[width=0.45\textwidth]{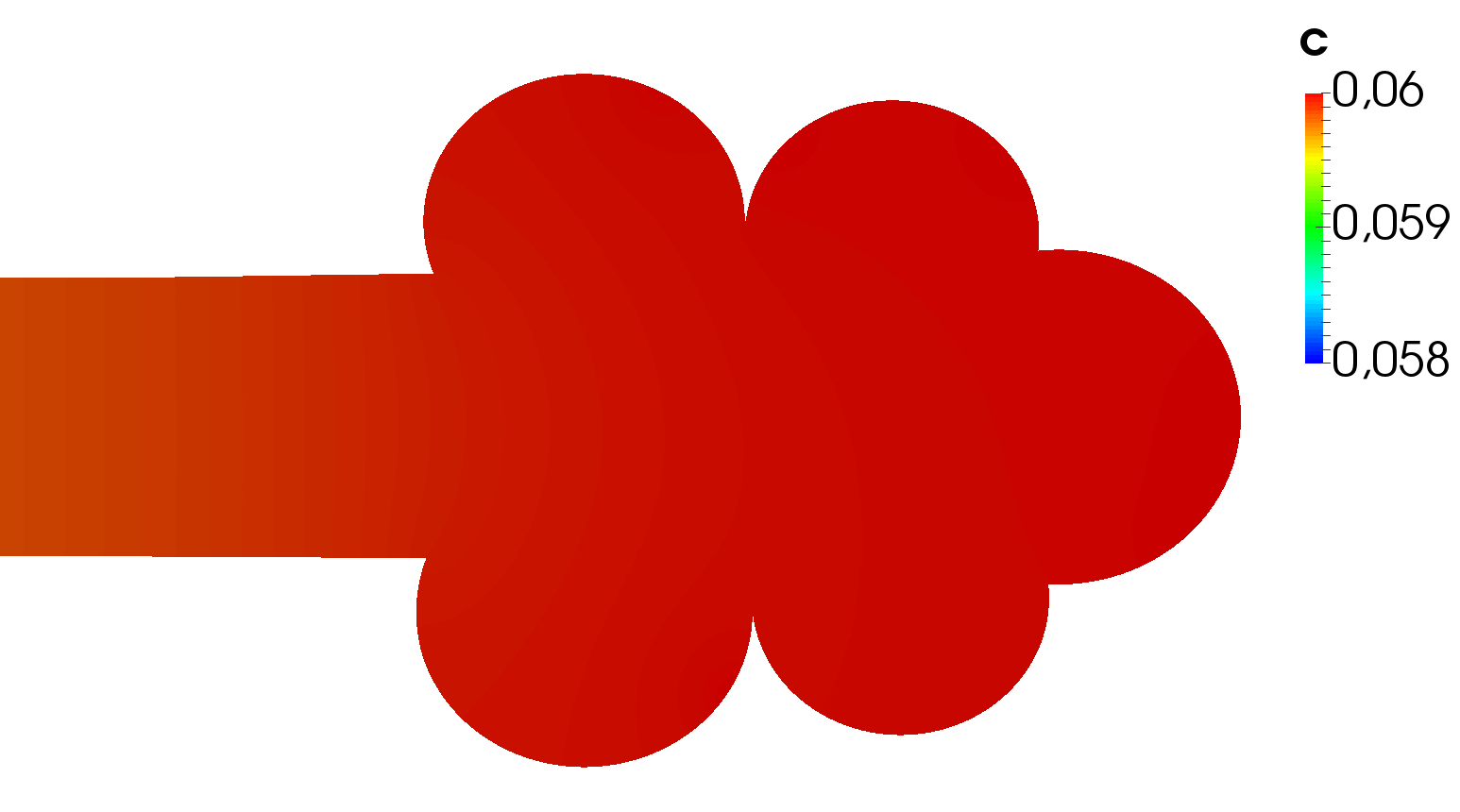}
 \caption{\label{fig.shortduct} CO$_2$ concentrations $c$ at times $t=2.4$s (top left), $t=2.5$s (top right) and $t=2.6s$ (bottom) for an alveolar sac consisting 
 of 5 alveoli with 6 permeable channels to the cardiovascular system. 
  A (relatively) short alveolar duct has been added on the left end of the domain.
 After 2.5s the boundary condition changes from inflow to outflow, which has a large influence on the results.
 At $t=2.4s$ (top) the concentration is maximal (red colour), where the boundary parts $\Gamma_{bl}$ are located.
}
\end{figure}

In fact the boundary condition for the concentration $c$ imposed on $\Gamma_{io}$ has a strong influence on the CO$_2$ concentration in $\Omega(t)$. 
The practical problem is the imposition of the external concentration $c_{\text{ext}}$ in the alveolar duct, especially
right after the expiration period ($t=5$s), when a relatively large concentration of CO$_2$ has just
left through $\Gamma_{io}$.
As a work-around, we consider longer artificial alveolar ducts of length $l=2mm$ and $l=10mm$, respectively and suppose that the exterior concentration is 
$c_{\text{ext}}=0.04302$ after this extension.
We compare the results for the three alveolar ducts in Figure~\ref{fig.CO2}.

\begin{figure}[bt]
\centering
\begin{minipage}{0.47\textwidth}
\includegraphics[width=\textwidth]{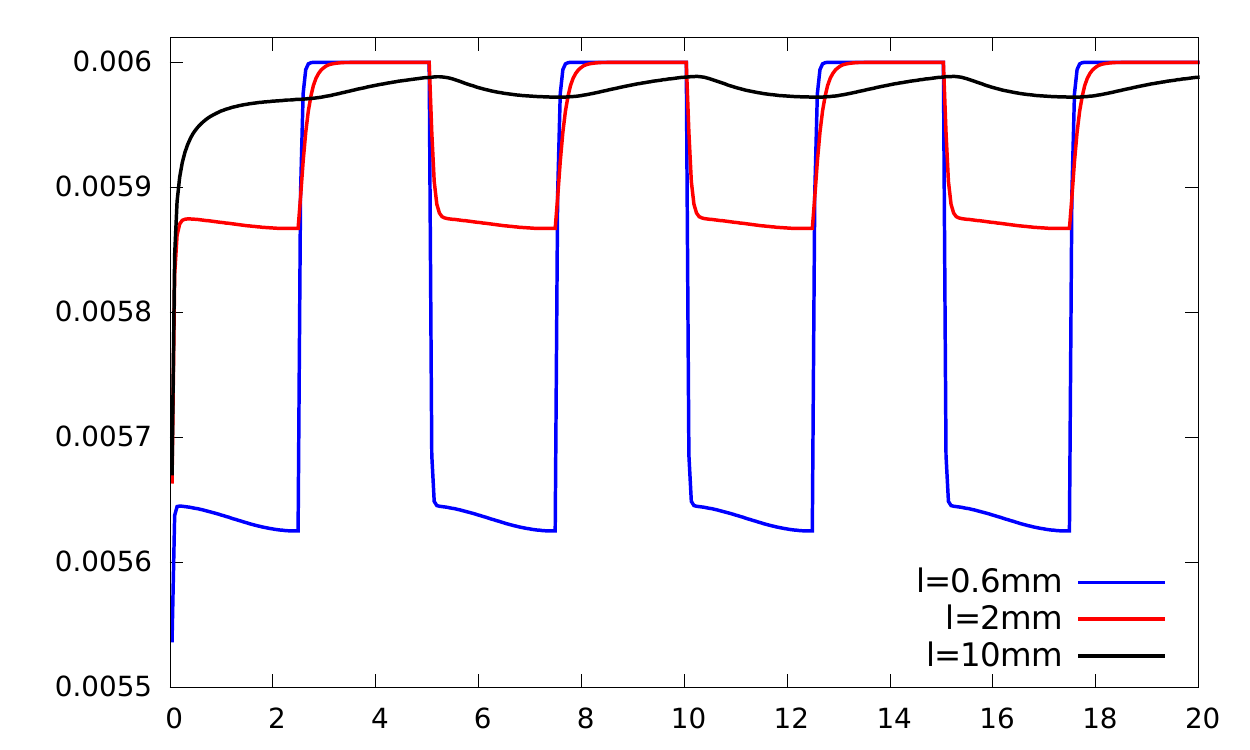}
\end{minipage}
\hspace{-0.4cm}
\begin{minipage}{0.47\textwidth}
\includegraphics[width=\textwidth]{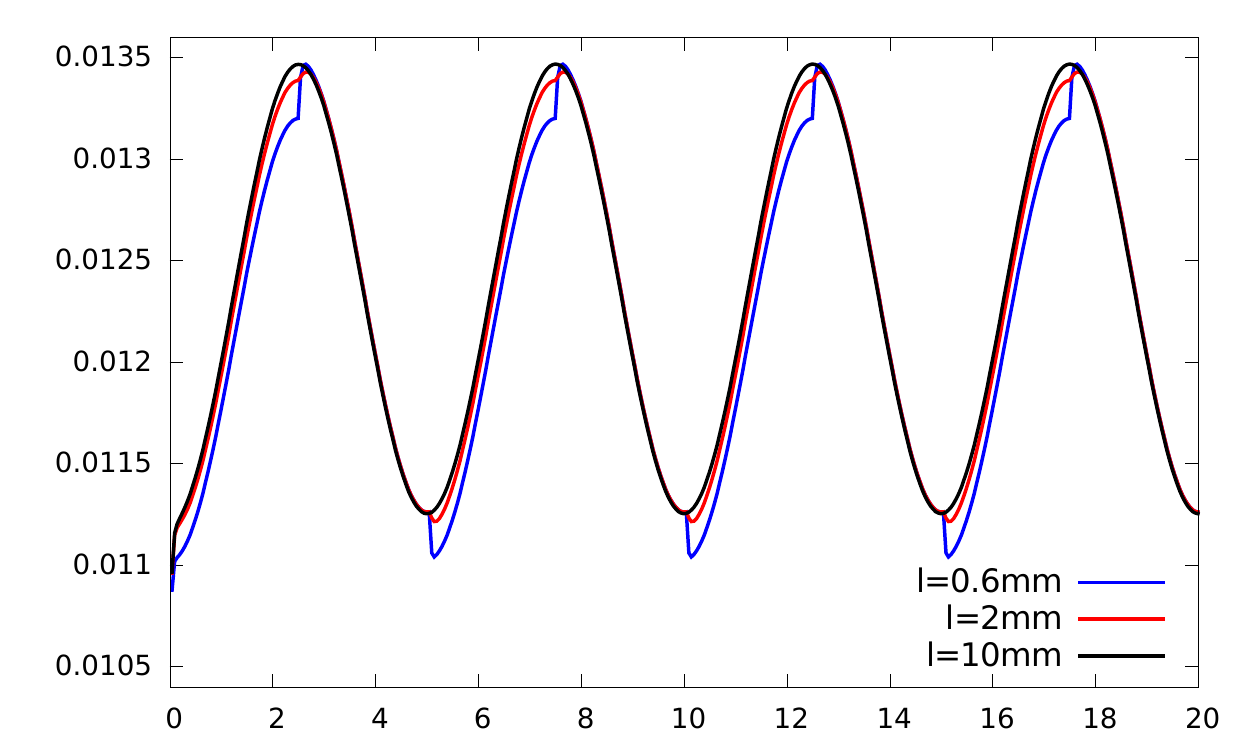}
\end{minipage}

 \caption{\label{fig.CO2} Comparison of the average concentration $J_{\Gamma_0}$ of CO$_2$ on the vertical entrance line $\Gamma_0$ (\textit{left}) 
 and of the total amount $J_{\Omega}$ of CO$_2$ in the 
 alveolar sac (\textit{right}) over time depending on the length of the artificial alveolar duct using the classical boundary conditions and Nitsche's method.}
\end{figure}

In the left plot of Figure~\ref{fig.CO2}, we show the average concentration of CO$_2$ on the vertical line $\Gamma_0$
between the alveolar sac and the alveolar duct
\begin{align*}
 J_{\Gamma_0}(c_h):= \frac{1}{|\Gamma_0|} \int_{\Gamma_0} c_h \, do, \qquad \Gamma_0:=\{(x,y)\in \Omega\, | \, x=0 \}. 
\end{align*}
Using the short alveolar duct ($l=0.6mm$), this value jumps between a value between 0.056 and 0.057 during inflow to a value of 0.06 during outflow. 
The reason is that in the outflow period, the concentration is
almost constant in the alveolar sac due to the large diffusion (see Figure~\ref{fig.shortduct}, bottom). In the inflow period, on the other hand, it varies between 
$0.04302$ on $\Gamma_{io}$ and $0.06$. 
These jumps can also be observed in the total amount of 
CO$_2$ in the alveolar sac $\Omega_a$
\begin{align*}
 J_{\Omega}(c_h) =\int_{\Omega_0(t)} c_h \, dx = \int_{\hat{\Omega}_0} \hat{J} \hat{c}_h \, d\hat{x}, \qquad \Omega_0:=\{(x,y)\in \Omega\, | \, x\geq0 \},
\end{align*}
see Figure~\ref{fig.CO2}, right plot. A similar behaviour, but with smaller jumps can be observed for the duct with length $l=2mm$. Here the average concentration on the 
line $\Gamma_0$ jumps between $J_{\Gamma_0}(c_h)\approx 0.0588$ and $0.06$.
For the longest alveolar duct ($l=10mm$), we observe a smooth behaviour of both functionals. 
The average value on $\Gamma_0$ oscillates here smoothly between 0.0597 and 0.0599. In conclusion, we see that the CO$_2$ concentration depends strongly on the boundary condition
and the length of the alveolar duct. It is difficult to guess, how long the artificial duct has to be to get realistic results.

In Figure~\ref{fig.artificial} we show the same functionals using the artificial boundary conditions. To see that this boundary condition is (almost) independent of the position,
at which the alveolar duct is cut, we compare results for ducts of length $l=0.1mm$, $0.6mm$ and $2mm$. First, we note that the curves for the total amount of CO$_2$ $J_{\Omega}$
are almost identical. The curves for the concentration on the line $\Gamma_0$ look very similar as well, especially for the smaller alveolar ducts. The small deviations are due to 
the fact, that the exact ``transparent'' boundary conditions are only approximated. Comparing the values with the results for the classical boundary conditions, 
the curves are much smoother with values varying between $0.05998$ and $0.06$. The larger values during inspiration are 
due to the fact that the large diffusion is also considered in the case of inflow.

\begin{figure}[bt]
\centering
\begin{minipage}{0.47\textwidth}
\includegraphics[width=\textwidth]{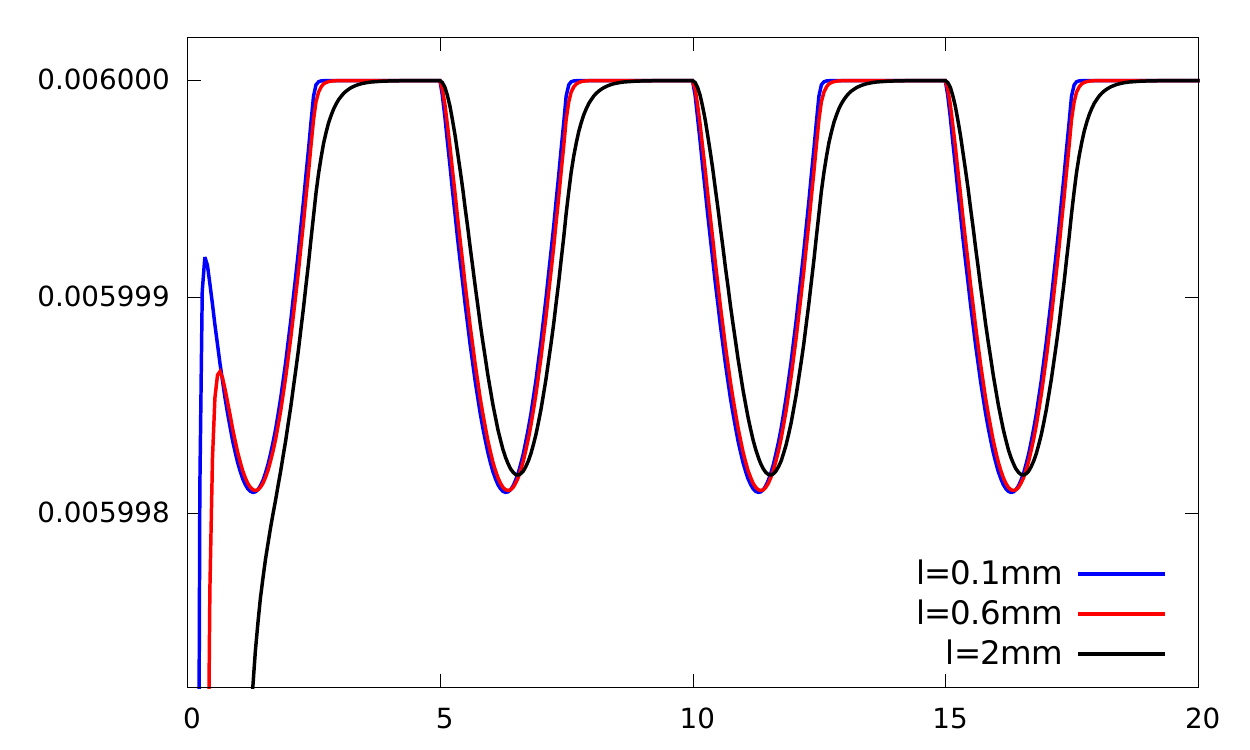}
\end{minipage}
\hspace{-0.4cm}
\begin{minipage}{0.47\textwidth}
\includegraphics[width=\textwidth]{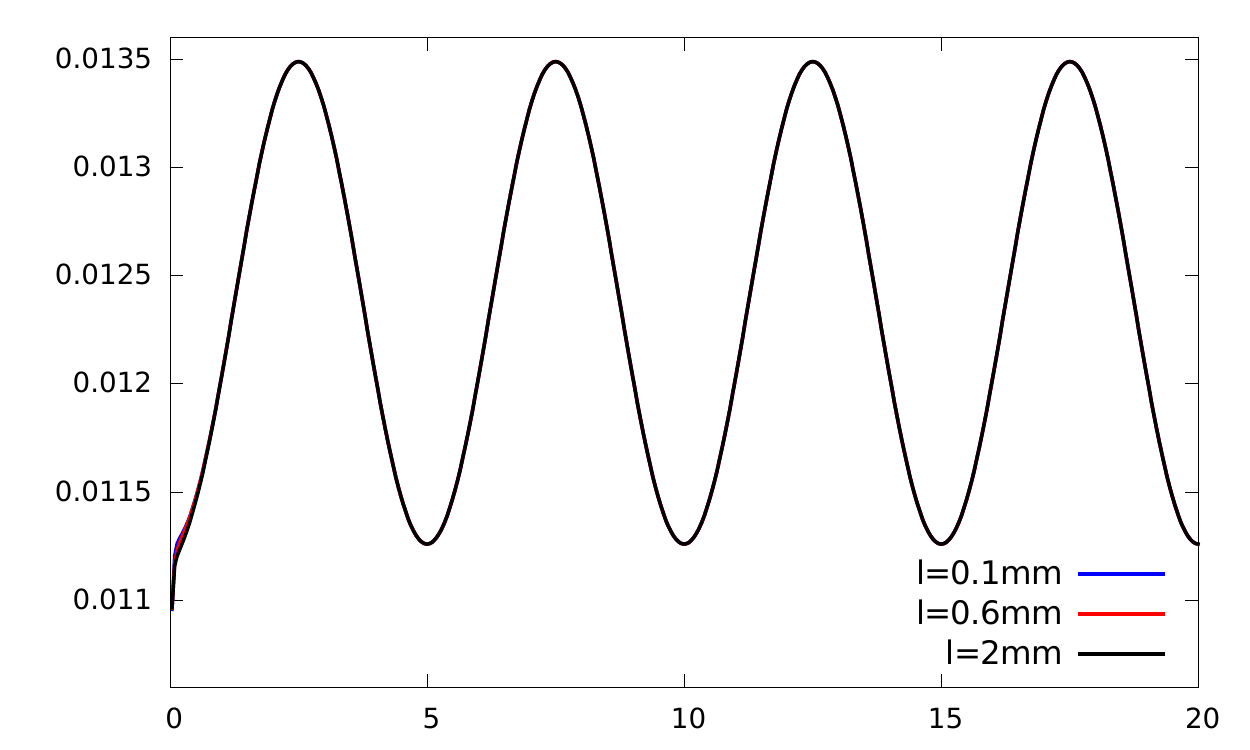}
\end{minipage}

 \caption{\label{fig.artificial} Comparison of the average concentration $J_{\Gamma_0}$ of CO$_2$ on the vertical entrance line $\Gamma_0$ (\textit{left}) 
 and of the total amount $J_{\Omega}$ of CO$_2$ in the 
 alveolar sac (\textit{right}) over time depending on the length of the artificial alveolar duct using the artificial boundary conditions.}
\end{figure}

Finally we would like to remark that in both cases the length of the alveolar duct does not have a significant influence on the fluid quantities: 
While the velocity fields in the alveolar 
sac are almost identical,
the pressure varies by an additive constant (as the do-nothing condition on its left end $\Gamma_{io}$ implies $\int_{\Gamma_{io}} p_h \, do =0$).

\paragraph{Influence of the domain movement}
Next, we want to investigate the influence of the domain movement on the gas concentrations and on the fluid forces. 
Therefore, we repeat the same calculation on a fixed domain, i.e.$\,$ setting $\hat{T}=$id, but keeping the Dirichlet values $v=v^{\text{dom}}$
as above.
In Figure~\ref{fig.fixed}, we compare the total amount $J_{\Omega}$ of CO$_2$ in the  
alveolar sac, as well as the wall shear stress on $\Gamma_{bl}(t)$
\begin{align*}
 J_{\sigma,\Gamma_{bl}} := \int_{\Gamma_{bl}(t)} \sigma n e_1 \, do,
\end{align*}
which is important when studying pathologies such as pulmonary emphysema. 
The amount of CO$_2$ shows large differences. While it varies by almost 20 \% over time on the moving domain, the variation is less than 0.1\% on a 
fixed domain. Note, however, that most of this variation comes from the change of volume of the domain itself and for example the amount of CO$_2$ over a fixed line,
e.g. $\Gamma_0$, shows much less variations. On the other hand, the wall stresses
differ considerably as well, its maximum being around 10\% higher when considering an expanding alveolar sac.

\begin{figure}[bt]
\centering
\begin{minipage}{0.47\textwidth}
 \includegraphics[width=\textwidth]{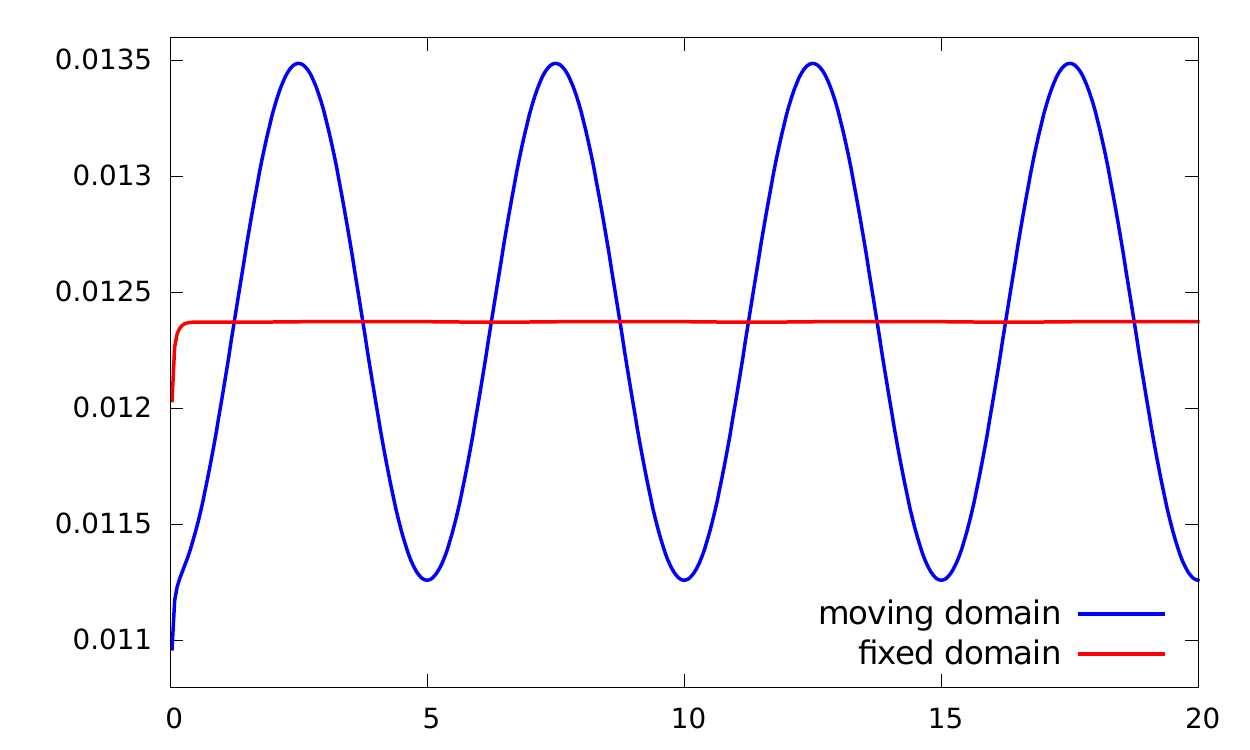}
 \end{minipage}
 \hspace{-0.4cm}
 \begin{minipage}{0.47\textwidth}
  \includegraphics[width=\textwidth]{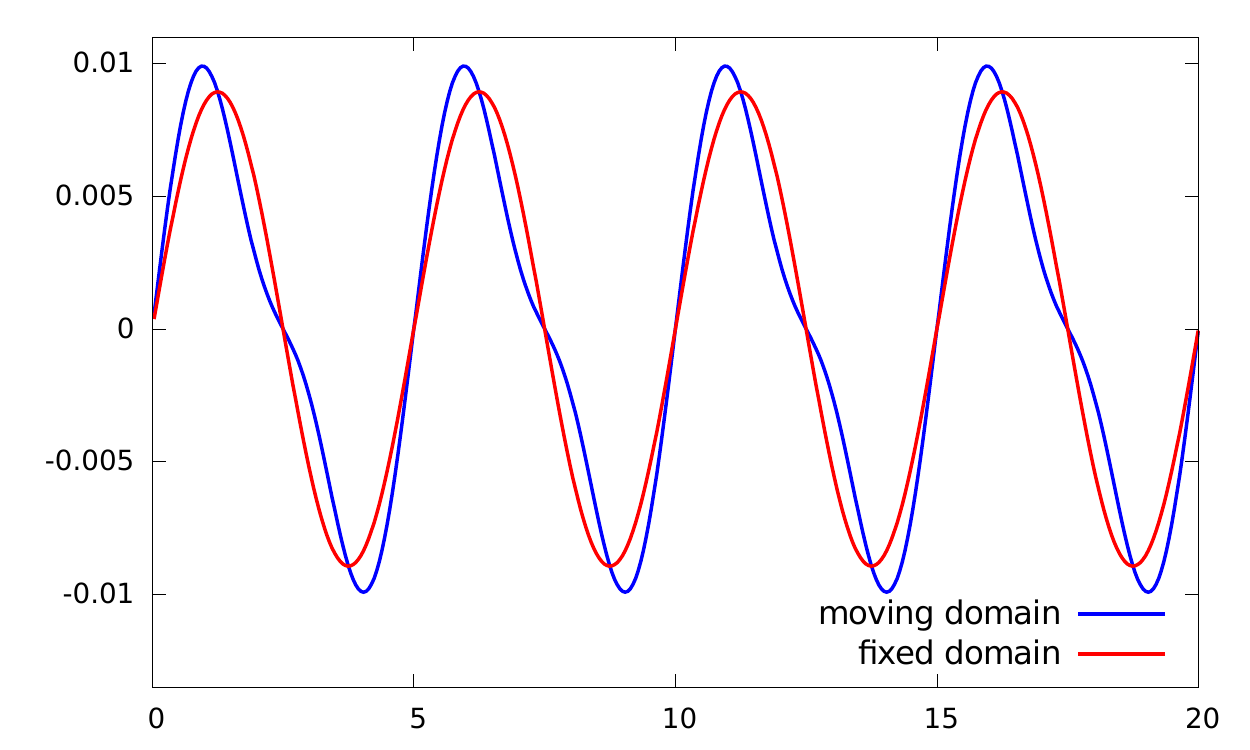}
  \end{minipage}
 \caption{\label{fig.fixed} Comparison of simulation results on a moving domain $\Omega(t)$ and on a fixed domain $\Omega$
  using artificial boundary conditions and a short alveolar duct of length $l=0.6mm$.
 \textit{Left}: Amount $J_{\Omega}$ of CO$_2$ within the alveolar sac over time.
 \textit{Right}: Wall stress $J_{\sigma,\Gamma_{bl}}$ over time.}
\end{figure}

\paragraph{Convergence analysis}

We close this section by analysing convergence properties of some functionals in space in Table~\ref{tab.conv}. We show values of
the vorticity functional
\begin{align*}
 J_{\text{vort}}(v_h) &= \int_{\Omega(t)} \left(\partial_y v_{h,1} - \partial_x v_{h,2}\right)^2 \, dx,
 \end{align*}
 the $L^2$-norm of the pressure as well as the functional $J_{\Omega}(c_h)$ 
at time $t=8.75$s on different meshes and for both pressure stabilisation techniques, while using the 
artificial boundary conditions. This instant of time within the second expiration period has been chosen, 
as both the vorticity and the pressure functional attain a maximum at $t=8.75$s.
In Table~\ref{tab.conv}, we observe a nice convergence behaviour for all three functionals and for both pressure stabilisations. 
 We have calculated an extrapolated value $j_e$ and an estimated convergence rate $\alpha_{\text{conv}}$ by fitting the function 
 $f(h) = j_e + ch^{\alpha_{\text{conv}}}$ with parameters $j_e, c$ and $\alpha_{\text{conv}}$ against the functional values on the three finer grids.
The estimated convergence orders are for all three functionals significantly higher than what can be expected for $Q_1$ finite elements. We would like 
to remark, however, that these might be slightly over-estimated, as the reference values are extrapolated from the same functional values 
we compute the errors for.

Comparing the results for the two pressure stabilisation techniques, we see that the pressure converges faster for the LPS stabilisation, while the opposite
holds true for the vorticity functional $J_{\text{vort}}(v_h)$. The reason for the faster convergence of $J_{\text{vort}}(v_h)$ for the edge-oriented stabilisation
might be the consistency of this method, which has the effect that the velocity error is better separated from the error in the pressure. The values of $J_{\Omega}(c_h)$
are nearly independent (i.e.$\,$identical in the first 10 digits) of the chosen pressure stabilisation.

\begin{table}[t!]
\centering
\begin{tabular}{r|cc|cc}
&\multicolumn{2}{c|}{$S_{\text{LPS}}$} &\multicolumn{2}{c}{$S_e$}\\
\#cells & $J_{\text{vort}}(v_h)$ &$\big|J_{\text{vort}}(v_h) - j_{\text{e}}\big|$ & $J_{\text{vort}}(v_h)$ &$\big|J_{\text{vort}}(v_h) - j_{\text{e}}\big|$\\ 
\hline
 656  &\;$2.71774\cdot10^{-2}$\; &$1.52\cdot 10^{-3}$ &$2.78808\cdot10^{-2}$\;  &$7.96\cdot 10^{-4}$\\
2624  &\;$2.84055\cdot10^{-2}$\; &$2.87\cdot 10^{-4}$ &$2.85460\cdot10^{-2}$\;  &$1.31\cdot 10^{-4}$\\
10496 &\;$2.86756\cdot10^{-2}$\; &$1.68\cdot 10^{-5}$ &$2.86736\cdot10^{-2}$\;  &$3.53\cdot 10^{-6}$\\
41984 &\;$2.86915\cdot10^{-2}$\; &$9.86\cdot 10^{-7}$ &$2.86770\cdot10^{-2}$\;  &$9.51\cdot 10^{-8}$\\
\hline 
 $j_\text{e}/\alpha_{\text{conv}}$&$2.86925\cdot10^{-2}$ &4.09 &$2.86771\cdot 10^{-2}$ &5.21
\end{tabular}
\vspace{0.3cm}

\begin{tabular}{r|cc|cc}
&\multicolumn{2}{c|}{$S_{\text{LPS}}$} &\multicolumn{2}{c}{$S_e$}\\
\#cells &$\|p_h\|_{\Omega_0}$ &$\big|\|p_h\|_{\Omega_0}-j_{\text{e}}\big|$ &$\|p_h\|_{\Omega_0}$ &$\big|\|p_h\|_{\Omega_0}-j_{\text{e}}\big|$ 
\\
\hline
 656  &$2.08528\cdot10^{-4}$\;  &$1.28\cdot 10^{-6}$   &$2.09178\cdot 10^{-4}$ & $1.90\cdot 10^{-6}$\\
2624   &$2.07492\cdot10^{-4}$\;  &$2.45\cdot 10^{-7}$  &$2.07824\cdot 10^{-4}$ & $5.42\cdot 10^{-7}$\\
10496  &$2.07283\cdot10^{-4}$\;  &$3.64\cdot 10^{-8}$  &$2.07448\cdot 10^{-4}$ & $1.66\cdot 10^{-7}$\\
41984  &$2.07252\cdot10^{-4}$\;  &$5.41\cdot 10^{-9}$  &$2.07333\cdot 10^{-4}$ & $5.10\cdot 10^{-8}$\\
\hline 
 $j_\text{e}/\alpha_{\text{conv}}$&$2.07247\cdot 10^{-4}$ &2.75 &$2.07282\cdot 10^{-4}$ &1.70\\
\end{tabular}
\vspace{0.3cm}

\begin{tabular}{r|cc}
&\multicolumn{2}{c}{$S_{\text{LPS}}/S_e$}\\
\#cells &$J_{\Omega}(c_h)$ &$\big|J_{\Omega}(c_h) - j_{\text{e}}\big|$\\
\hline
656   &$1.23987\cdot10^{-2}$ &$8.44\cdot 10^{-6}$\\
2624  &$1.23668\cdot10^{-2}$ &$2.35\cdot 10^{-5}$\\
10496 &$1.23884\cdot10^{-2}$ &$1.88\cdot 10^{-6}$\\
41984 &$1.23901\cdot10^{-2}$ &$1.51\cdot 10^{-7}$\\
\hline
 $j_\text{e}/\alpha_{\text{conv}}$ &$1.23903\cdot 10^{-2}$ &3.64 
\end{tabular}
 \caption{\label{tab.conv} Functional values and estimated errors for the functionals $J_{\text{vort}}(v_h), \|p_h\|_{L^2(\Omega_0)}$ and $J_{\Omega}(c_h)$ on different 
 meshes using anisotropic LPS and edge stabilisation. 
 The functional values are compared against an extrapolated reference value $j_e$ and a convergence rate $\alpha_{\text{conv}}$
 is estimated based on the values on the three finer grids.
 The values of $J_{\Omega}(c_h)$ were identical in the first 10 digits for both pressure stabilisation techniques. In the last row of each table, we show 
 for each functional the
 extrapolated value $j_e$ in the first column and the estimated convergence order $\alpha_{\text{conv}}$ in the second column.
 }
\end{table}

\section{Conclusion and outlook}
\label{sec:conclusion}

We have presented a numerical framework for the simulation of gas flow at small scale in the alveolar sacs of the human lung, including the
gas exchange with the cardiovascular system. We have considered that the alveolar sacs are deformable and have introduced numerical stabilisation terms, that
are able to handle even extreme domain deformations. Moreover, we have shown that the choice of boundary conditions has a significant impact on the results
and that the classical in- and outflow boundary conditions might not be a good choice on the alveolar scale.

The present work can be seen as a first step to model gas flow and gas exchange at an alveolar scale in the human lung.
A complete model would have to consider both the small alveolar scale as well as the lung in total at a larger scale.

A model for the alveolar scale is needed to study different pathologies, for example pulmonary emphysema, where foreign substances, e.g. tobacco, might block
the membrane towards the cardio-vascular system. Moreover, the properties of the alveolar wall itself can be altered by foreign substances, which can reduce the
deformability of the membrane. A full fluid-structure interaction problem, including the interaction with the thin alveolar wall and
possibly a further fluid or solid model for the space around the lung,
has to be considered. For a first work in a similar direction, we refer to Dailey \& Ghadali~\cite{DaileyGhadali}.

\begin{acknowledgement}
We gratefully acknowledge financial support by CONCYTEC Peru within the program 
``Programa nacional de innovaci\'on para la competitividad y productividad'' (PNICP, 361-PNICP-PIBA-2014)
 as well as travel support by the Heidelberg Graduate School
of Mathematical and Computational Methods for the Sciences (HGS MathComp). The second author was supported 
 by the DFG Research Scholarship FR3935/1-1.
\end{acknowledgement}

\bibliographystyle{plainnat}


\end{document}